\documentclass[a4paper,11pt,reqno]{amsart}
\usepackage{amsmath}
\usepackage{amsthm,enumerate}
\usepackage{mathtools}
\usepackage{enumitem}
\usepackage{graphicx}
\usepackage{amssymb}
\usepackage{soul}
\usepackage[toc,page]{appendix}

\usepackage[utf8]{inputenc}

\usepackage{color}
\usepackage{url}

\usepackage[text={6.5in,8.6in},centering]{geometry}

\usepackage{srcltx}

\usepackage[T1]{fontenc}

\theoremstyle{plain}
\newtheorem{thm}{Theorem}[section]
\theoremstyle{plain}
\newtheorem{lem}[thm]{Lemma}
\newtheorem{assum}[thm]{Assumption}
\newtheorem{prop}[thm]{Proposition}
\newtheorem{cor}[thm]{Corollary}

\theoremstyle{definition}
\newtheorem{defi}{Definition}[section]
\newtheorem{rem}{Remark}[section]

\newcommand{\R}{\mathbb{R}}
\newcommand{\rn}{\mathbb{R}^{N}}

\newcommand{\rnn}{\mathbb{R}^{N}}
\newcommand{\hn}{M}
\newcommand{\hnn}{\mathbb{H}^{N}}

\renewcommand{\c}{\mathsf{c}}

\newcommand{\RR}{\mathbb{R}}
\newcommand{\HH}{\mathbb{H}}

\newcommand{\GM}{\mathbb{G}_{M}^s}

\newcommand{\GMcc}{\mathbb{G}_{M_{\mathsf{c}}}^s}

\newcommand{\GR}{\mathbb{G}_{\RR^N}^s}
\newcommand{\GH}{\mathbb{G}_{\HH^N}^s}

\newcommand{\dx}{\,{\rm d}x}

\newcommand{\rd}{{\rm d}}

\newcommand{\dg}{\rd\mu_{\hn}}

\newcommand{\dgcc}{\rd\mu_{M_\c}}
\newcommand{\dgh}{\rd\mu_{\HH^N}}

\numberwithin{equation}{section} \allowdisplaybreaks

\definecolor{darkblue}{rgb}{0.05, .05, .65}
\definecolor{darkgreen}{rgb}{0.1, .65, .1}
\definecolor{darkred}{rgb}{0.8,0,0}

\begin{document}
 \title[Fractional Porous Medium Equation on manifolds]{The Fractional Porous Medium Equation\\ on noncompact Riemannian manifolds}
\author{Elvise Berchio}
\address{\hbox{\parbox{5.7in}{\medskip\noindent{Dipartimento di Scienze Matematiche, \\
Politecnico di Torino,\\
        Corso Duca degli Abruzzi 24, 10129 Torino, Italy. \\[3pt]
        \em{E-mail address: }{\tt elvise.berchio@polito.it}}}}}

\author{Matteo Bonforte}
\address{\hbox{\parbox{5.7in}{\medskip\noindent{Departamento de Matem\'aticas,   Universidad Aut\'onoma de Madrid, and \\
ICMAT - Instituto de Ciencias Matem\'{a}ticas, CSIC-UAM-UC3M-UCM,\\
Campus de Cantoblanco, 28049 Madrid, Spain.
 \\[3pt]
        \em{E-mail address: }{\tt matteo.bonforte@uam.es}}}}}

\author{Gabriele Grillo}
\address{\hbox{\parbox{5.7in}{\medskip\noindent{Dipartimento di Matematica,\\
Politecnico di Milano,\\
   Piazza Leonardo da Vinci 32, 20133 Milano, Italy. \\[3pt]
        \em{E-mail address: }{\tt
          gabriele.grillo@polimi.it}}}}}

          \author{Matteo Muratori}
\address{\hbox{\parbox{5.7in}{\medskip\noindent{Dipartimento di Matematica,\\
Politecnico di Milano,\\
   Piazza Leonardo da Vinci 32, 20133 Milano, Italy. \\[3pt]
        \em{E-mail address: }{\tt
          matteo.muratori@polimi.it}}}}}

\date{}

\dedicatory{Dedicated to the Memory of Marek Fila, Mathematician and Friend}

\keywords{Fractional porous medium equation; fractional Laplacian; noncompact manifolds; a priori estimates; curvature bounds; Sobolev inequality; smoothing effects}

\subjclass[2010]{Primary: 35R01. Secondary: 35K65, 35A01, 35R11, 58J35.}

\begin{abstract}
We study nonnegative solutions to the Fractional Porous Medium Equation on a suitable class of connected, noncompact Riemannian manifolds.  We provide existence and smoothing estimates for solutions, in an appropriate weak (dual) sense, for data belonging either to the usual $L^1$ space or to a considerably larger weighted space determined in terms of the fractional Green function. The class of manifolds for which the results hold includes both the Euclidean and the hyperbolic spaces and even in the Euclidean situation involves a class of data which is larger than the previously known one.
\end{abstract}

\maketitle

\normalsize

\section{Introduction} Let $M$ be an $N$-dimensional   geodesically   complete, connected, noncompact   Riemannian   manifold. We shall study \emph{nonnegative} solutions to the Fractional Porous Medium Equation (FPME):
\begin{equation}\label{NFDE}
\left \{ \begin{array}{ll}
\partial_t u + (- \Delta_{\hn})^{s} \!\left(u^m\right)= 0 &  (t,x)\in  (0, \infty) \times \hn\,,\\
u(0,x)=u_0(x)\ge0\,, & x\in \hn,
\end{array}
\right.
\end{equation}
where  $(-\Delta_{\hn})^s$ denotes the (spectral) fractional Laplacian on $M$, defined as the spectral $s$-th power of the  Laplace-Beltrami operator $ -\Delta_M $ associated to the Riemannian metric of $M$,  $0 < s < 1$,  $m > 1$. We investigate well-posedness of \eqref{NFDE} and the validity of suitable \it smoothing effects \rm for such evolution, this meaning quantitative bounds on the $L^\infty$ norm of the solution at time $t>0$ in terms of a (possibly weighted) $L^1$ norm of the initial datum. These smoothing effects will be proved for data that belong to a space which is \it strictly larger \rm than  $L^1(M)$ , and will involve an  $L^1$ -norm which is \it weighted \rm w.r.t the $s$-Green function of $M$, namely to the kernel of  $(-\Delta_M)^{-s}$ , which is of course assumed to exist. This will be a consequence of our main hypotheses, in particular of Assumption \ref{general} on $M$. It should be stressed that the class of data we shall deal with is \it larger \rm than  {the previously known one}   in the Euclidean case, in fact we allow in that case for (nonnegative) data $u_0$ that are dominated by $ C |x|^{-a}$  for all $|x|\geq R$, for some $C, R>0$ and $a>2s$, which hence need not be integrable.

The FPME in the Euclidean setting has been the object of extensive research, starting from the seminal works \cite{dPQRV1, dPQRV2, BV3}. A number of subsequent works, see \cite{BFR, BFV, BSV, BV2, BV1}, involve the behaviour of analogues of \eqref{NFDE} when the equation is posed in a bounded Euclidean domain, when \it suitable versions \rm of homogeneous Dirichlet boundary conditions are assumed. In these papers the authors develop a common setup that allows to deal with many different nonlocal operators, including the three most common non-equivalent versions of the  (Dirichlet) fractional Laplacian: the  restricted (or standard), the spectral and the censored (or regional). All of them require different kind of homogeneous Dirichlet boundary conditions. A common feature of dealing with such problems consists in relating the behaviour of solutions to the analogue of \eqref{NFDE} to properties of the \it Green function \rm of the operator defining the evolution, for which precise \it pointwise estimates \rm are explicitly known.  This method is quite flexible and allows to deal with different settings, as it encodes the properties of the evolution in the   Green functions's behaviour  .  It can also be used to extend the theory to a larger class of data, belonging to a weighted $L^1$ space, the weight being the fractional   Green   function.

On the other hand, so far this kind of method has been used in settings in which the   Green function's   behavior is fairly well-known, in the form of two-sided pointwise bounds. Therefore, it is a nontrivial task to generalize it to cases in which the   Green function's   pointwise behavior is not precisely known. This kind of difficulty already appeared in the paper \cite{BBGG}, where equation \eqref{NFDE} is studied only in the \it hyperbolic space\rm, an important but specific case.  There, the fractional Green function can be estimated explicitly by means of sharp heat kernels estimates, obtaining clean estimates about its tail behavior.
Related results, obtained by different methods, have appeared in \cite{RS} (see also \cite{RS2, RS3} for previous investigation in related, yet different geometric settings), where existence and suitable smoothing effects are proven, for initial data belonging to $L^1\cap L^p$ with $p\ge2$. The strategy adopted in \cite{RS} consists in exploiting suitable fractional versions of Gagliardo-Nirenberg-Sobolev or log-Sobolev inequalities, and then run a nowadays standard Moser-type iteration. We refer to \cite{DA} for related results in the linear setting, or to \cite{dPQRV2, GMPo} for the nonlinear case, and to \cite{V2} for a general treatment as nonlinear diffusions are concerned. It seems not possible to adapt the latter methods to deal with the larger class of data that we consider here.

We mention that nonlinear (local) diffusions of porous medium type on manifolds have been the object of several recent papers, see e.g. \cite{BGV, GM1, GM, GMP, GMV, GMV-MA, V4}, where a quite complete basic theory has been developed: existence, uniqueness, smoothing effects, behaviour of fundamental solutions, and the large time behaviour of nonnegative solutions.
Most of those papers deal with the case of Cartan-Hadamard manifolds (i.e. manifolds that are complete, simply connected and whose sectional curvature is everywhere nonpositive), the prototype being the hyperbolic space. Significant differences with the Euclidean framework are shown, e.g. in the form of the smoothing effects, long-time asymptotics of solutions, different exponent range for extinction in finite time.
It seems not easy to extend the ``local techniques'' of these paper to the present nonlocal (fractional) setting, hence new strategies are necessary. We show here that the above mentioned Green function methods, together with geometric analysis tools, allow to prove existence of nonnegative solutions and smoothing effects for a large class of data, see Theorem \ref{smoothPhi}. Moreover, when dealing with  manifolds whose sectional curvature is bounded above by a strictly negative constant, we can show \it improved smoothing effects for $L^1$ data, \rm \eqref{thm.smoothing.HN-like.estimate.2} in Theorem \ref{thm-existence}, which proves a fractional analogue of the local result proven in \cite{GM, V4}.

We conclude this introduction by noticing that we shall make no use of the extension methods originally introduced in \cite{CS}, and then shown to hold in several significant geometric situations in \cite{BGS} (see also \cite{BGFW} for related topics), relying instead only on the spectral definition of the fractional Laplacian. It is possible that extension methods allow to prove smoothing effects in an alternative way, see e.g. \cite{dPQRV2, GMP1}. 

The paper is organized as follows. In Section \ref{PMR} we list our notations, geometric assumptions and definitions, and state our main results, Theorems \ref{thm-existence} and \ref{thm.smoothing.HN-like}. Section \ref{potential} provides some crucial and delicate technical lemmas, namely comparison results for integrals of the Green function, and comparison between potentials of test functions and the Green function, that will be essential ingredients in our main proofs. Section \ref{init-data} carefully discusses the class of data we are dealing with. In Section \ref{smooth} we prove existence and fundamental estimates for approximate solutions, whereas the proof of our main results is fully given in Section \ref{sect.semigroups}. The short Section \ref{open} lists some open problems that might originate future lines of research.

 \section{Preliminaries and statements of the main results}\label{PMR}

In the following, through some geometric and functional assumptions, we first describe the general Riemannian setting in which we set our problem. Subsequently, we provide the precise notion of solution to \eqref{NFDE} we will work with, and then state our related main results regarding existence and smoothing effects.

\subsection{Notation}  If $ M $ is the Riemannian manifold at hand, we let $ \dg  $ denote its volume measure (sometimes written as $ \dg(x) $ when it is relevant to highlight the integration variable), $ \mathrm{Ric}(M) $ its Ricci curvature and $ \mathrm{sec}(M) $ its sectional curvature. For any $ x_0 \in M $ and $ r>0 $, the standard symbol $ B_r(x_0)$ stands for the geodesic ball of radius $ r $ centered at $ x_0 $. The geodesic distance between two points $ x,y \in M $, often interpreted as a radius, is denoted by $ r(x,y) $ or $ r_M(x,y) $ in order to avoid ambiguity when working with different manifolds.

Since we will have to deal with several multiplying constants, whose exact value is immaterial to our purposes, we have decided to use as much as possible the general symbol $C$. The actual value may therefore change from line to line, without explicit reference. However, when it is significant to specify the dependence of $C$ on suitable parameters, we will write it explicitly. Nevertheless, in some cases, in order to avoid ambiguity, we will use other symbols.

\subsection{Geometric assumptions and consequences}\label{GAC}
 We will assume the following primary condition on $M$, which is crucial to most of our results.
\begin{assum}\label{general}
	We require that $M$ is an $N$-dimensional ($N\geq 2$) complete, connected, noncompact Riemannian manifold such that its Ricci curvature is bounded below:
 \begin{equation}\label{ric}
  {\rm Ric}(M)\geq -(N-1)k \quad  \text{ for some } k>0\,.
 \end{equation}
Besides, we require that the following Faber-Krahn inequality holds:
\begin{equation}\label{FK}
 \lambda_1(\Omega)\ge c \, \mu_M(\Omega)^{-\frac2N}
\end{equation}
for a suitable $c>0$, where $\Omega$ is an arbitrary open, relatively compact subset of $M$ and $\lambda_1(\Omega)$ is the first eigenvalue of the Laplace-Beltrami operator $-\Delta_M$ with homogeneous Dirichlet boundary conditions on $\partial\Omega$.
\end{assum}

We point out that \eqref{FK} is equivalent to the \emph{Nash inequality}
\begin{equation*}
\left\| f \right\|_2^{1+\frac2N} \le C \left\| f \right\|_1^{\frac2N}  \left\| \nabla f \right\|_2
\end{equation*}
 and, when $N\ge3$, to the \emph{Sobolev inequality}
\[
\left\| f \right\|_{\frac{2N}{N-2}} \le C \left\| \nabla f \right\|_2 ,
\]
for all smooth and compactly supported $f$ (see e.g.~\cite[Chapter 8]{H} and \cite{C}).

Furthermore, it is also known that \eqref{FK} implies, for all $\varepsilon>0$, $x,y\in M$ and $t>0$, the following \emph{Gaussian upper bound} on the \emph{heat kernel} $k_{\hn}(t,x,y)$ of $M$:
\begin{equation}\label{gaussian}
k_{\hn}(t,x,y)\le \frac C{t^\frac N2} \, e^{-\frac{r(x,y)^2}{(4+\varepsilon)t}} \, ,
\end{equation}
where $ C>0 $ is a suitable constant that depends only on $ N,c,\varepsilon $. It is also important, to the sequel, to observe that \eqref{FK} implies (see e.g.~\cite[formula (8.4)]{H}) the following lower bound for volumes of balls:
\begin{equation}\label{non-coll}
\mu_M(B_r(x_0))\ge C\, r^N  \qquad  \text{for all $ r>0 $ and all $ x_0\in M $} \, ,
\end{equation}
for a suitable constant $C>0$ depending only on $ N $ and $ c $. On the other hand, it is a standard consequence of \eqref{ric} (through the Bishop-Gromov theorem, see e.g.~\cite[Theorem 1.1]{H}) that the reverse inequality holds for small radii:
\begin{equation}\label{uppervolume}
\mu_M(B_r(x_0)) \le C\,r^N \qquad  \text{for all $ 0<r<1 $ and all $ x_0\in M $} \, ,
\end{equation}
for another constant $C>0$ depending only on $ N $ and $ k $.

Let us now briefly investigate the consequences of the above assumptions in terms of the \emph{fractional Laplacian}, namely the operator $ (-\Delta_M)^s $ defined as the spectral $s$-th power of the  Laplace-Beltrami operator $ -\Delta_M $,   meant as the Friedrichs extension of the (nonnegative and essentially self-adjoint, see \cite[Theorem 2.4]{Stri}) same operator acting on test functions. Its $s$-th power is then defined via the spectral theorem.  Note that it is also given by the explicit formula
$$
 (-\Delta_M)^s f(x) = \frac1{\Gamma(-s)}\int_0^{+\infty} \left(\int_M k_{\hn}(t,x,y)\left(f(y)-f(x)\right)
 \, \dg(y)\right)\, \frac{\rm{d}t}{t^{1+s}}
$$
for a suitable set of functions $f$. For instance, this is the case for every $ f  $ in the $L^2(M)$ domain of $ (-\Delta) $ (see e.g.~\cite[Theorem 4.3.5]{J}), thus in particular smooth and compactly supported functions will do. In this regard, let us also mention \cite{Sti} and references quoted therein, in particular \cite{B,K}, or the recent preprints \cite{Caselli, Caselli2}.

Thanks to the Gaussian bound \eqref{gaussian}, we have in particular that $M$ is \emph{$s$-nonparabolic}, in the sense that the integral
\begin{equation}\label{fract}
\GM(x,y):=\int_0^{+\infty}\frac{k_{\hn}(t,x,y)}{t^{1-s}}\,{\rm{d}t}
\end{equation}
is finite for all $x,y\in M$ with $x\not=y$. It is well known that the function $\GM$ defined above is the \it fractional Green function \rm on $M$ (or fractional Riesz potential), in the sense that $ (-\Delta_M)^s \, \GM(\cdot,y)  = \delta_y $ for every $ y \in M $, where $ \delta_y $ stands for the Dirac delta centered at $y$ (such identity should be understood in the distributional sense). More precisely, from \eqref{gaussian} and \eqref{fract}, a straightforward computation yields the Euclidean-type bound
\begin{equation}\label{green-euc}
\GM(x,y)\le \frac C{r(x,y)^{N-2s}} \qquad \forall x,y\in M \, ,
\end{equation}
for some $ C=C(N,c,s)>0 $. Under stronger (curvature) assumptions such as the ones we will introduce below, we are actually able to provide some sharper estimates (see   \eqref{Hyp.Green.RN2} and \eqref{Hyp.Green.HN0b}).

Note that, in view of the continuity of the map $ x \mapsto k_{\hn}(t,x,y) $ for every fixed $ (t,y) \in \mathbb{R}^+ \times M $, by virtue of estimate \eqref{gaussian} we can apply Lebesgue's dominated convergence theorem to \eqref{fract} and obtain that for every fixed $ y \in M $, the map $ x \mapsto \GM(x,y) $ is continuous in $ \hn \setminus \{ y \} $.

Once the fractional Green function has been shown to be well defined, one can introduce, for any sufficiently regular function $ \psi $, its \emph{fractional potential} by means of the following ``convolution'' formula:
\begin{equation}\label{fract-bis}
(-\Delta_M)^{-s} \psi (x) : = \int_{\hn} \psi(y) \, \GM(x,y) \, \dg(y) =  \int_0^{+\infty} \left(\int_M\frac{k_{\hn}(t,x,y)}{t^{1-s}}\psi(y)\, \dg(y)\right)\, {\rm{d}t}\,.
\end{equation}
The symbol $ (-\Delta_M)^{-s} $ is not employed by chance, since it turns out to be the true left-inverse operator of $ (-\Delta_M)^s $, at least on appropriate subspaces of functions (for more details see Lemma \ref{lemma-inv} below). We will often deal with potentials of \emph{nonnegative bounded and compactly supported} functions, for which Assumption \ref{general} guarantees that $ (-\Delta_M)^{-s} $ is both well defined and it satisfies suitable two-sided bounds, see Lemma \ref{boundgreen} below.

In some of our results, we will need the following stricter assumptions.

\begin{assum}\label{ch}
We require that $M$ is an $N$-dimensional Cartan-Hadamard manifold, namely that $M$ is complete, simply connected and has everywhere nonpositive sectional curvature.
\end{assum}

Note that if $ M $ is a Cartan-Hadamard manifold, then \eqref{FK} is always true (see again \cite[Chapter 8]{H}), whereas \eqref{ric} should still be required separately.

In order to obtain some improved estimates, we will require a stricter hypothesis: 

\begin{assum}\label{neg}
We require that $M$ is an $ N $-dimensional Cartan-Hadamard and, besides, that
\begin{equation}\label{sec}
  {\rm sec}(M)\le  -\c \qquad \text{for a given } \c >0\, .
 \end{equation}
 \end{assum}

\subsection{Definition of Weak Dual Solutions (WDS)}\label{def-wds}
We will deal with suitable solutions to \eqref{NFDE} starting from initial data that belong to the classical $L^1(\hn)$ space or to the following \emph{weighted space}, defined in terms of the fractional Green function:
\begin{equation*}\label{LG}
L^1_{\GM}(\hn) := \left \{ f : M \rightarrow \R \text{ measurable} : \ \sup_{x_0\in M}   \left\| f \right\|_{L^1_{x_0,\GM}} < +\infty  \right\} ,
\end{equation*}
where, for every fixed $ x_0 \in \hn $, we put
\begin{equation}\label{normLGx0}
\left\| f \right\|_{L^1_{x_0,\GM}}  := \int_{B_1(x_0)} \left| f (x) \right| \dg(x) + \int_{\hn \setminus B_1(x_0)} \left| f(x) \right| \GM(x ,x_0) \, \dg(x) \, .
\end{equation}
The space $ L^1_{x_0,\GM}(M) $ is in turn defined as the set of all measurable functions for which the norm in \eqref{normLGx0} is finite. The space $L^1_{\GM}(\hn)$ is thus a reinforcement of the latter, since we require that $ \| \cdot \|_{L^1_{x_0,\GM}} $ is \emph{uniformly} bounded with respect to the reference point $x_0 $. It is therefore natural to endow $L^1_{\GM}(\hn)$ with the norm
\begin{equation*}
\left\| f \right\|_{L^1_{\GM}}  := \sup_{x_0\in M}  \left\| f \right\|_{L^1_{x_0,\GM}} .
\end{equation*}
For usual $ L^p(M) $ spaces, the corresponding norm will typically be written as $ \| \cdot \|_{L^p(\hn)} $, except in some cases where for readability purposes we will adopt the more compact notation $ \| \cdot \|_{p} $.

Under Assumption \ref{general}, thanks to \eqref{green-euc}, it is apparent that $ \GM(x,x_0) \leq C$ for all $x \in \hn\setminus B_1(x_0)$ and all $x_0\in \hn$, so that the inclusion $L^1(\hn) \subseteq L^1_{\GM}(\hn)$ holds. Moreover, the inclusion $ L^1_{\GM}(\hn) \subseteq L^1_{x_0,\GM}(\hn) $ trivially holds by definition. One may then wonder whether those spaces actually coincide. The answer is negative, as in Section \ref{init-data} we will provide some explicit examples showing that such inclusions are \emph{strict}: $L^1(\hn) \subsetneq L^1_{\GM}(\hn) \subsetneq L^1_{x_0,\GM}(\hn) $ for all $x_0\in \hn$. In addition, on space forms (i.e.~$ \RR^N $ or $\HH^N$), we will also provide admissible decay rates for functions to belong to $L^1_{\GM}(\hn)$, which somehow give a more noticeable feeling  about how larger these spaces can be compared to $ L^1(M) $.

We are now ready to give a proper definition of \emph{weak dual solution}, based on the simple (formal) observation that applying the operator $ (-\Delta_M)^{-s} $ to both sides of \eqref{NFDE}, we obtain the ``dual equation''
\begin{equation}\label{dual-idea}
\partial_t \! \left[ (- \Delta_{\hn})^{-s} u \right] + u^m = 0 \, .
\end{equation}

\begin{defi}\label{defi_WDS}
	
 Let $ u_0 \in L^1_{\GM}(\hn) $, with $ u_0 \ge 0 $. We say that a nonnegative measurable function $ u  $ is a Weak Dual Solution (WDS) to problem~\eqref{NFDE} if, for every $T>0$:

\begin{itemize}
	
\item $u \in C^0([0, T];  L^1_{x_0,\GM}(\hn))$ for all $ x_0 \in \hn $;

\smallskip

\item  $u^m \in L^1( (0, T) ; L^{1}_{loc}(\hn) ) $;

\smallskip

\item $u$ satisfies the identity
\begin{equation}\label{def_eq}
\int_{0}^{T} \int_{\hn}  \partial_t \psi \,  (- \Delta_{\hn})^{-s} u \, \dg \, {\rm d}t - \int_{0}^{T} \int_{\hn} u^m \, \psi \, \dg\, {\rm d}t = 0
\end{equation}
for every test function $\psi \in C^1_c((0,T); L_c^{\infty}(\hn))$;

\smallskip

\item $u(0,\cdot)=u_0$ a.e.~in $\hn$.

\end{itemize}

\end{defi}

\begin{rem}\label{remdef}
Equation \eqref{def_eq} is well defined under Assumption \ref{general}. More precisely, we can show that $  (- \Delta_{\hn})^{-s} u \in C^0([0,T];L^1_{loc}(M)) $. Indeed, for any $ \sigma>0 $ and $ x_0 \in \hn $, thanks to Lemma \ref{lem.Phi.estimates} below (plus Remark \ref{generalR}) we can assert that
$$(-\Delta_{\hn})^{-s} \left( \chi_{B_\sigma(x_0)} \right)\!(x) \leq  C \qquad \text{for all $x \in M$: }  r(x,x_0)<1 $$
and
$$(-\Delta_{\hn})^{-s} \left( \chi_{B_\sigma(x_0)} \right)\!(x) \leq C \, \GM(x, x_0) \qquad \text{for all $x \in M$: }  r(x,x_0) \ge 1 \, , $$
for a suitable constant $ C>0 $. In particular, formula \eqref{fract-bis} and Fubini's theorem yield (recall that $ \GM(x,y)=\GM(y, x) $), for all $t \in [0,T]$:
$$
\begin{gathered}
\int_{B_\sigma(x_0)}  (-\Delta_{\hn})^{-s}u (t,x) \, \dg(x) = \int_{\hn}  u(t,x) \, (-\Delta_{\hn})^{-s}\left(  \chi_{B_\sigma(x_0)}  \right)\!(x) \, \dg(x) \le C \left\| u(t) \right\|_{L^1_{x_0,\GM}} \, .
\end{gathered}
$$
Given the arbitrariness of $ \sigma $ and $x_0 $, we thus deduce that $ t \mapsto (-\Delta)^{-s} u(t,\cdot) $ is a curve with values in $ L^1_{loc}(M) $. Moreover, since $ t \mapsto u(t,\cdot) $ is a \emph{continuous} curve with values in $ L^1_{x_0,\GM}(M) $, still by applying the above estimates and Fubini's theorem to differences we obtain for all $t,s \in [0,T]$
$$
\begin{aligned}
 \int_{B_\sigma(x_0)} &\left|  (-\Delta_{\hn})^{-s}u (t,x) - (-\Delta_{\hn})^{-s}u (s,x) \right| \dg(x) \\
\le & \, \int_{B_\sigma(x_0)} (-\Delta_{\hn})^{-s}\left( \left|  u (t,\cdot) - u(s,\cdot) \right|\right)\!(x) \, \dg(x) \\
= & \, \int_{\hn} \left| u(t,x)-u(s,x) \right| (-\Delta_{\hn})^{-s}\left(  \chi_{B_\sigma(x_0)}  \right)\!(x) \, \dg(x) \le C \left\| u(t)-u(s)  \right\|_{L^1_{x_0,\GM}}\,, 
\end{aligned}
$$
which yields the claim.
\end{rem}

\subsection{Main results}
We will construct a WDS for any nonnegative initial datum $ u_0\in  L^1_{\GM}(\hn)$, as a monotone limit of nonnegative $L^1$ semigroup (mild) solutions, see Section \ref{smooth} for further details.
\begin{thm}[Existence of a WDS for data in $L^1_{\GM}$]\label{thm-existence}
	Let $M$ satisfy Assumption \ref{general}, and let $u_0 $ be any nonnegative initial datum such that $ u_0\in  L^1_{\GM}(\hn)$. Then there exists a weak dual solution to problem~\eqref{NFDE}, in the sense of Definition \ref{defi_WDS}.
\end{thm}

In this paper we do not address \emph{uniqueness} issues, but it is worth noticing that our WDS with data in $ L^1_{\GM}(M) $ are obtained as monotone limits of mild solutions in $ L^1(M) \cap L^\infty(M)$, generated by a monotone sequence of initial data. Hence, as in \cite{BV1}, it is not difficult to show that within this subclass, solutions are unique:
\begin{cor}[Uniqueness of limit WDS]\label{unique}
The WDS $u$ constructed in Theorem \ref{thm-existence} as a monotone limits of mild $ L^1(M) \cap L^\infty(M)$ solutions, does not depend on the particular choice of the monotone approximating sequence of initial data.
\end{cor}\vspace{-2mm}

Let us define the exponent
\begin{equation*}
\vartheta_1:=\frac{1}{2s+N(m-1)} \, ,
\end{equation*}
and state our $L^1$-$L^{\infty}$ smoothing estimates.

\begin{thm}[Smoothing effects for data in $ L^1 $]\label{thm.smoothing.HN-like}
Let $ \hn $ satisfy Assumption \ref{general}. Let $u$ be the WDS to \eqref{NFDE}, constructed in Theorem \ref{thm-existence}, corresponding to any nonnegative initial datum $u_0 \in L^1(\hn)  $. Then there exists $ C = C (N,k,c,s,m)>0$ such that
\begin{equation}\label{S-1}
\left\| u(t) \right\|_{L^\infty(M)} \le C \left( \frac{\left\| u(t) \right\|_{L^1(M)}^{2s \vartheta_1}}{t^{N \vartheta_1}} \vee \left\| u_0 \right\|_{L^1(M)} \right) \le C \left( \frac{\left\| u_0\right\|_{L^1(M)}^{2s \vartheta_1}}{t^{N \vartheta_1}} \vee \left\| u_0\right\|_{L^1(M)} \right) \qquad \forall t > 0 \, .
\end{equation}	

If, in addition, $M$ satisfies Assumption \ref{ch}, then for some $ C=C(N,s,m)>0 $ we have
\begin{equation}\label{S-2}
\left\| u(t) \right\|_{L^\infty(M)} \le C  \, \frac{\left\| u(t) \right\|_{L^1(M)}^{2s \vartheta_1}}{t^{N \vartheta_1}}  \le C \, \frac{\left\| u_0\right\|_{L^1(M)}^{2s \vartheta_1}}{t^{N \vartheta_1}}  \qquad \forall t > 0 \, ,
\end{equation}	
Furthermore, if $M$ also satisfies Assumption \ref{neg} (and $ u_0 \not \equiv 0 $), then
\begin{equation}\label{thm.smoothing.HN-like.estimate.2}
\left\| u(t) \right\|_{L^\infty(\hn)}  \leq \frac{C}{t^{\frac{1}{m-1}}} \left[ \log\!\left(t \left\| u_0 \right\|_{L^1(\hn)}^{m-1} \right)\right]^{\frac{s}{m-1}} \qquad \forall t \ge  e^{(N-1)(m-1)\sqrt{\c}} \left\| u_0 \right\|_{L^1(\hn)}^{-(m-1)}  ,
\end{equation}
for another $C=C(N,s,\c,m)>0$.
\end{thm}

  \begin{rem}\begin{itemize}
\item  In the case where $ \hn $ is a Cartan-Hadamard manifold, in fact Assumption \ref{general} in Theorem \ref{thm.smoothing.HN-like} is unnecessary. Indeed, the $L^1$-semigroup theory provides a suitable solution to \eqref{NFDE}, which belongs to $ C^0([0,T];L^1(M)) $, regardless of functional or curvature assumptions on $ M $. The latter come into play when one wants to ``invert'' the equation and pass to a dual formulation as in \eqref{dual-idea}. However, the validity of the crucial bounds \eqref{Hyp.Green.RN2} and \eqref{Hyp.Green.RN}, which are for free on any Cartan-Hadamard manifold, is all one needs in order to prove existence of WDS and the above smoothing effects when $ u_0 \in L^1(\hn) $.
\item An analogue of the bound \eqref{S-2}, involving the $L^p$ norm of the initial datum for $p\ge2$, is shown in \cite{RS}, for a smaller class of data, namely those in $L^1\cap L^p$ for some $p\ge2$, but under more general assumptions on $M$. 
\item The long-time decay of solution provided by \eqref{thm.smoothing.HN-like.estimate.2} is faster than the Euclidean one, Assumption \ref{neg} being crucial for this to hold. That for large time a faster decay overtakes the one valid in the Euclidean framework is an effect of negative curvature, which somehow increases the speed of propagation and produces a better decay rate (see e.g.~\cite{GM,GMV,GMV-MA,V4}).

    \end{itemize}
\end{rem}

When enlarging the class of allowed initial data, i.e.~when dealing with the space $L^1_{\GM}(M)$ in place of $ L^1(M) $,  we obtain the following $L^1_{\GM}$-$L^{\infty}$ smoothing estimates.

\begin{thm}[Smoothing effects for data in $L^1_{\GM}$]\label{smoothPhi} Let $M$ satisfy Assumption \ref{general}.
Let $u$ be the WDS to \eqref{NFDE}, constructed in Theorem \ref{thm-existence}, corresponding to any nonnegative initial datum $u_0 \in L^1_{\GM}(M) $. Then there exist $ C_1=C_1(N,k,c,s,m)>0$ and $ C_2=C_2(N,k,c,s,m)>0 $ such that  	
\begin{equation}\label{thm.smoothing.HN-like.estimate3}
\left\| u(t) \right\|_{L^\infty(M)} \le C_1 \left( \frac{\left\| u(t) \right\|_{ L^1_{\GM} }^{2s \vartheta_1}}{t^{N \vartheta_1}} \vee \left\| u_0 \right\|_{ L^1_{\GM} } \right) \le C_2 \left( \frac{\left\| u_0\right\|_{ L^1_{\GM} }^{2s \vartheta_1}}{t^{N \vartheta_1}} \vee \left\| u_0\right\|_{ L^1_{\GM} } \right) \qquad \forall t > 0 \, .
\end{equation}
If, in addition, $M$ satisfies Assumption \ref{ch} (and $ u_0 \not \equiv 0 $), then there exists $ C_3=C_3(N,k,s,m)>0 $ such that
\begin{equation}\label{thm.smoothing.HN-like.estimate4}
\left\| u(t) \right\|_{L^\infty(M)} \le C_3 \, \dfrac{\left\| u_0 \right\|_{L^1_{\GM}}^{\frac1m}}{t^{\frac1m}} \qquad \forall t \ge \left\| u_0 \right\|_{{L^1_{\GM}}}^{-(m-1)} .
\end{equation}
\end{thm}

  \begin{rem}
\noindent\begin{itemize}
\item  The bounds of Theorem \ref{smoothPhi} are new even in ${\mathbb R}^N$. It is remarkable that the short time bound \eqref{thm.smoothing.HN-like.estimate3} shows the same time dependence as the one established in the $L^1$ setting in \cite{dPQRV2}.
\item The above theorems apply also   in the prototype   case of the hyperbolic case. Theorem \ref{smoothPhi} fixes an error in Theorem 2.4 of \cite{BBGG}, in which the bound \eqref{thm.smoothing.HN-like.estimate4} is incorrectly stated for all times, by providing the right short-time behaviour. It also provides the correct functional space associated to the Green function there. We also comment that a similar remark on the functional space applies also to Theorem 2.3 of \cite{BBGG}, although we shall not pursue analogues of that result here.

\end{itemize}
\end{rem}

\section{Fractional green function and potential estimates}\label{potential}
In this section we collect some fundamental estimates for fractional Green functions and potentials, which will have a crucial role when dealing with WDS.

\subsection{Comparison between Green functions and between integrals of Green functions}
We establish key inequalities comparing the Green function on $\hn$ (and its integrals over geodesic balls), subject to the curvature bound \eqref{sec}, with the Green function on the associated space form $M_{\c}$ (i.e.~the hyperbolic space of constant curvature $ -\c $).
\begin{lem}\label{Lemma.ComparisonGreen}
Let $M$ satisfy Assumption \ref{neg}, and let $M_{\c}$ be the $N$-dimensional space form of constant (sectional) curvature equal to $-\c$, $\mu_{M_\c}$ its volume measure and $\GMcc$ its fractional Green function. Then, for all $r>0$ and all $o\in M$, we have
\begin{equation}\label{comp1}
\int_{B_r(o)}\GM(x,o)\, \dg(x)\le
\int_{B_r(o_\c)}\GMcc(x,o_\c)\, \dgcc(x)\,,
\end{equation}
where  $o_\c $ stands for any pole in $ M_\c $ and $ B_r(o_\c) \subset M_\c $ for the geodesic ball of radius $ r $   centered   at $o_\c$. Furthermore, we also have that
\begin{equation}\label{comp2}
\GM(x,y)\le  \GMcc(x_\c,y_\c)
\end{equation}
for all $x,y\in M$ and   all $x_\c,y_\c\in M_\c$ such that $r_M(x,y)=r_{M_\c}(x_\c,y_\c)$. 
\end{lem}
\begin{proof}
Let us consider the following Cauchy problem for the heat equation on $ \hn $:
\begin{equation}\label{eqM}
\begin{cases}
\partial_t u =\Delta_M u & \mbox{ in }  {(0,+\infty) \times M} \, , \\
u(0,\cdot)=\chi_{B_r(o)} & \mbox{ in }M \, .
\end{cases}
\end{equation}
The solution of   this   problem is the function $U_M$ given by the explicit formula
\begin{equation*}\label{solM}
U_M(t,x)=\int_M k_M(t,x,y) \, \chi_{B_r(o)}(y)\,{\rm d}\mu_M(y) = \int_{B_r(o)}k_M(t,x,y)\,{\rm d}\mu_M(y) \, ,
\end{equation*}
where we recall that $k_M$ is the (minimal) heat kernel on $M$. Besides, in view of \eqref{fract}, we have that
\begin{equation}\label{intgreen}\begin{aligned}
\int_{B_r(o)}\GM(y,o)\,{\textrm d}\mu_M(y)& =\int_{B_r(o)}\left(\int_0^{+\infty}\frac{k_M(t,y,o)}{t^{1-s}} \, {\rm d}t\right) \, {\rm d}\mu_M(y)\\
&=\int_0^{+\infty}\frac1{t^{1-s}}\left(\int_{B_r(o)}k_M(t,y,o)\,{\rm d}\mu_M(y)\right)\, \rm{d}t
\end{aligned}\end{equation}
by Fubini's theorem (note that all the functions involved are positive and the integrals involved are finite, since $M$ is $s$-nonparabolic).

Now we observe that the quantity $\int_{B_r(o)}k_M(t,y,o)\,{\rm d}\mu_M(y)$ appearing in the last integral on the r.h.s.~of \eqref{intgreen} coincides with $U_M(t,o)$ (upon recalling that $ k_M $ is symmetric w.r.t.~space variables). We will now compare such solution for all $x$, hence for $x=o$ as well, with the  {radial}   solution of a suitable problem posed in $M_{\c}$. More precisely, let us consider problem \eqref{eqM} with $M$ is replaced by $M_\c$ and $o$ replaced by $o_c$, and according to the above notation let $U_{M_\c}$ denote the corresponding solution. We claim that $U_{M_\c}$, transplanted on $M$,   that is
$$
U_{M_\c} \equiv U_{M_\c}(t,r_M(x,o)) \, ,
$$
is a \emph{supersolution} to \eqref{eqM}. This will be a direct consequence of the \emph{Hessian comparison theorem} and the fact that $U_{M_\c}(t,\cdot)$ is nonincreasing as a function of the geodesic distance $r_{M_\c}(x,o_\c )$ on $M_\c$ (see e.g.~\cite[Subsection 2.2]{GMV}). What follows aims at proving such delicate monotonicity property.

Firstly, it is readily seen that $U_{M_\c}(t,\cdot)$ is a function of $r_{M_\c}$ only,  since the initial datum is radial by construction and $M_\c$ is a spherically-symmetric manifold (i.e.~a \emph{model manifold}). In order to show that $U_{M_\c}(t,\cdot)$ is also \emph{nonincreasing} as a function of $ r = r_{M_\c}$ we notice that, by the explicit expression of the metric on $M_\c$, the Laplacian evaluated on functions $f$ depending only on $r $ reads
\[
\Delta_{M_{\c}} f(r)=\frac{\partial^2f}{\partial r^2}+(n-1)\sqrt{\c} \, \coth\!\left(\sqrt{\c} r\right)\frac{\partial f}{\partial r} \, .
\]
Given $R_1, T>0$, consider now the solution $u$ to the Cauchy-Dirichlet problem
\begin{equation}\label{eee-1}
\begin{cases}
\partial_t u =\Delta_{M_{\c}} u &\quad \text{in }  {(0,T) \times B_{R_1}(o_\c)} \, , \\
u=0 & \quad \text{on }  {(0,T) \times \partial B_{R_1}(o_\c)} \, , \\
u(0,r)=u_0(r) & \quad \text{in }  B_{R_1}(o_\c)  \, ,
\end{cases}
\end{equation}
where $u_0$ is smooth, nonnegative, not identically zero, radial, nonincreasing and compactly supported in $ B_{R_1}(o_\c)  $. By standard parabolic theory, it follows that $u$ depends on $(t,r)$ only, is strictly positive in  {$(0,T) \times B_{R_1}(o_\c)$} and smooth in  {$(0,T) \times \overline{B}_{R_1}(o_\c)$}. Therefore, we have that  ${\partial_r u}(t,R_1) \le 0$   for all $t\in(0,T)$. Moreover, in view of the smoothness of $u$, we also have that   ${\partial_r u}(t,0)=0$   for all $t\in(0,T)$. A direct computation shows that the radial derivative $v := \partial_r u$ satisfies the parabolic equation
\[
\partial_t v =\Delta_{M_{\c}} v-\frac{\c(n-1)}{\sinh^2(\sqrt{\c}r)} \, v \qquad \text{for } (t,r) \in (0,T) \times (0,R_1)  \, ,
\]
and $v$ is nonpositive on the parabolic boundary because of the previous assumptions and considerations. The Schr\"odinger operator $-\Delta_{M_{\c}}+\frac{\c(n-1)}{\sinh^2(\sqrt{\c}r)}$, defined on its natural form domain (that takes into account the homogeneous Dirichlet boundary condition at $r=0$), generates a Markov semigroup for which, in particular, the {comparison principle} holds. As a result, we can infer that $ v \le 0  $ at all times, which means that $ r \mapsto   u(t,r) $ is nonincreasing for every $t\in(0,T)$. In fact, this can be rigorously justified by considering the problem in $(0,T) \times (\epsilon,R_1) $, noting that $v(t,\epsilon)\le c_\epsilon\to0$ uniformly in $t$ as $\epsilon\to0$ since $v$ is continuous down to $r=0$ because $u$ is smooth, and controlling $v$ via the supersolution $c_\epsilon$.

Finally, we pick an arbitrary $ R_1>r $ and take an increasing sequence $ \{ u_{0,n} \}$ of initial data which are smooth, nonnegative, not identically zero, radially nonincreasing and compactly supported in $ B_{r}(o_\c)  $ and converge pointwise to $ \chi_{B_r(o_\c)} $ as $ n \to \infty $. Let $ \{ u_{R_1,n} \} $ denote the corresponding sequence of  solutions to \eqref{eee-1}, with $ u_0 $ replaced by $u_{0,n} $. By what we have just shown, each $  u_{R_1,n} (t,\cdot)  $ is radially nonincreasing, and by standard comparison principles the family is monotone increasing with respect to both $ R_1 $ and $n$. Passing to the limit first as $ n \to \infty $ and then as $ R_1 \to + \infty $, we obtain (monotone) pointwise convergence to $ U_{M_\c}(t,\cdot)$, which therefore preserves such radial monotonicity properties (note that $ t <T $ but $T>0$ is arbitrary), and the proof of \eqref{comp1} is complete.

As concerns \eqref{comp2}, we just point out that, in view of \eqref{fract}, it is a straightforward consequence of the comparison between the heat kernels of $ \hn $ and $ \hn_\c $ (which holds due to \eqref{sec}, see e.g.~\cite[Theorem 4.2]{G}).   Note that, by assumption, the cut locus of any point in $M$ is empty, so that \cite[Theorem 4.2]{G} is indeed applicable.
\end{proof}

It is worth writing down more explicitly the consequences of Lemma \ref{Lemma.ComparisonGreen}, separating the Euclidean and hyperbolic cases.

$\bullet$ \textbf{Pseudo Euclidean case:} ${\rm sec}(M)\le 0$, that is, the manifold is allowed to have flat parts. In this case the associated space form is $M_\c=M_0=\RR^N$. Even though, for simplicity, we stated and proved Lemma \ref{Lemma.ComparisonGreen} only when $ \c > 0 $, it remains true for $ \c=0 $ as well, with inessential changes in the proof. Hence, \eqref{comp2} yields
\begin{equation}\label{Hyp.Green.RN2}
\GM(x,y)\le \mathbb{G}_{\R^N}^s(x_\c,y_\c) \le\frac{C}{|x_\c-y_\c|^{N-2s}}= \frac{C}{r(x,y)^{N-2s}} \qquad \forall x , y \in \hn : \, x \neq y \, ,
\end{equation}
whereas \eqref{comp1} yields
\begin{equation}\label{Hyp.Green.RN}
\int_{B_R(y)} \GM(x,y) \, \dg(x) \le  C \int_{B_R(y_\c)}\frac{1}{|x-y_\c|^{N-2s}} \, \dx = C R^{2s} \qquad \text{for all $ y \in \hn $ and all $ R>0 $}\, ,
\end{equation}
for some $C>0$ only depending on $N,s$.

\smallskip

$\bullet$ \textbf{Pseudo hyperbolic case:} ${\rm sec}(M)\le - \c$ for a given $\c >0$, that is, the manifold has everywhere negative curvature. Roughly speaking, in this case $\hn$ is allowed to have hyperbolic parts, but not Euclidean ones. Therefore, when $\c=1$ the associated space form is $M_\c=M_1=\HH^N$, while for a general $\c $ it is the $N$-dimensional hyperbolic space of constant sectional curvature $-\c$. By combining Lemma \ref{Lemma.ComparisonGreen} with \cite[Lemma 3.1 and Corollary A.2]{BBGG}, where proper estimates for  $\GH$ were derived, we can infer that there exists $C>0$, only depending on $N,s, \c$, such that
\begin{equation}\label{Hyp.Green.HN0}
\GM(x,y)\le \GMcc(x_\c,y_\c)\le \frac{C}{r_{\hn_\c}(x_\c,y_\c)^{N-2s}}=\frac{C}{r(x,y)^{N-2s}} \qquad \forall x,y \in \hn : \, x \neq y \, .
\end{equation}
However, when $r(x,y)\ge 1$, the latter estimate can significantly be improved as follows:
\begin{equation}\label{Hyp.Green.HN0b}
\GM(x,y)\le \GMcc(x_\c,y_\c)\le C\, \frac{ e^{-(N-1)\sqrt{ \c} \, r_{\hn_\c}(x_\c, y_\c)}}{r_{\hn_\c}(x_\c, y_\c)^{1-s}} = C\, \frac{ e^{-(N-1)\sqrt{ \c} \, r(x, y)}}{r(x, y)^{1-s}}  \qquad \forall x,y \in \hn : \, r(x,y) \ge 1 \, .
\end{equation}
As a consequence, direct computations in radial coordinates on the hyperbolic space show that
\begin{equation*}
\int_{B_R(y)} \GM(x,y) \,  \dg(x) \le \int_{B_R(y_\c)} \GMcc(x,y_\c) \, \dgcc(x) \le C R^{2s} \qquad \text{for all $ 0 < R \le 1 $} \, ,
\end{equation*}
and
\begin{equation}\label{Hyp.Green.HN2}
\int_{B_R(y)}\GM(x,y) \, \dg(x) \le  \int_{B_R(y_\c)} \GMcc(x,y_\c) \, \dgcc(x) \le C R^s \qquad \text{for all $ R \ge 1 $} \, ,
\end{equation}
for every $ y \in \hn $ and some $C>0$ only depending on $N,s, \c$.  {The first estimate follows from \eqref{Hyp.Green.HN0}, while the second is a consequence of \eqref{Hyp.Green.HN0b}, recalling that $ \dgcc(x) \equiv \sinh(\sqrt{\c} \, {r(x,y_\c)})^{N-1} \, \mathrm{d} r $.}

\subsection{Comparison between potentials and Green functions}
We now provide fundamental two-sided estimates, aimed at comparing the potential of a bounded and compactly supported function with the Green function itself (i.e.~the potential of a Dirac delta).

 \begin{lem}\label{lem.Phi.estimates}
Let $M$ satisfy Assumption \ref{general}. Let $ \psi \in \mbox{L}_c^{\infty}(\hn)$ be a nonnegative and nontrivial function such that $\operatorname{supp}( \psi)\subseteq B_{\sigma}(x_0)$ for some $0<\sigma<1$ and $x_0 \in M$. Then there exist two constants $\underline C=\underline C(N,k,c,s)>0$ and $ \overline C= \overline C(N,k,c,s)>0$ such that
\begin{equation}\label{boundgreen}
\underline C \left\| \psi \right\|_1 \left(1\wedge r(x_0,x)^{N-2s}\right) \GM(x, x_0) \leq  (-\Delta_{\hn})^{-s} \psi  (x) \leq \overline C\, \|\psi\|_{\infty} \, \sigma^N  \, \GM(x,x_0)  \quad \forall x\in \hn \setminus \{ x_0 \} \, .
\end{equation}
 \end{lem}
 \begin{proof}
We start by elaborating some bounds for the heat kernel $k_{\hn}$, that we will exploit in \eqref{fract-bis}. From Li-Yau estimates, see in particular \cite[Theorem 2.2(ii)]{LY} (note that, to our purposes, $ q=0 $ w.r.t.~their notations), we know that if $v$ is a positive solution to the heat equation on $ \hn $, then there exist constants $c_0,c_1,c_2,\beta>0$, which only depend on $N$ and $k $ in \eqref{ric}
such that
$$
v(t_1,x_1) \leq c_0 \left(\frac{t_2}{t_1} \right)^{\beta} v(t_2,x_2) \, e^{c_1\frac{r(x_1,x_2)}{t_2-t_1}+c_2(t_2-t_1)}
$$
for all $0<t_1<t_2<3$ and all $x_1, x_2 \in \hn$. Given any $x,y \in \hn $ and $ t \ge 2 $, this inequality applied to the solution $v(\cdot, \cdot)=k_{\hn}( \cdot +t-2,x,\cdot)$, with the choices $x_1=x_0$, $x_2=y$, $t_1=1$ and $t_2=2$, yields
$$
k_{\hn}(t-1,x,x_0) \leq c_0 \, 2^{\beta} \,  k_{\hn}(t,x,y) \, e^{c_1r(y,x_0)+c_2}\qquad \text{for all } t\geq 2 \text{ and all } x,y \in \hn \, ,
$$
which, in turn, entails
$$k_{\hn}(t-1,x,x_0) \leq C\, k_{\hn}(t,x,y) \qquad \text{for all } t\geq 2 \text{ and all $x,y \in\hn$: }  r(y,x_0) \leq \sigma \, , $$
for a suitable $C=C(N,k)>0$. Inserting such inequality into \eqref{fract-bis}, and recalling the support properties of $ \psi $, we get
\begin{equation}\label{4}
\begin{aligned}
 (-\Delta_{\hn})^{-s}  \psi (x) = & \, \int_{B_{\sigma}(x_0)} \psi(y) \int_0^{+\infty}  \frac{ k_{\hn}(t,x,y)}{t^{1-s}} \,{\rm{d}t } \, \dg(y) \\
  \geq & \,  C  \left\| \psi \right\|_{1} \int_1^{+\infty}  \frac{ k_{\hn}(t,x,x_0)}{t^{1-s}} \, {\rm{d}t } \qquad \forall x\in \hn \, ,
\end{aligned}
\end{equation}
for another  $ C>0 $ as above. In order to prove the lower bound in \eqref{boundgreen}, we need to appropriately estimate the rightmost integral in \eqref{4} for every $ x \in \hn $.

First we focus on the case $r(x,x_0)\geq 1$. Thanks to  \eqref{ric}, by \cite[Corollary 4.2 and Remark 4.4]{Sturm} the following estimate holds:
\begin{align}\label{6}
k_{\hn}(t,x,x_0) \geq \frac{c_5}{t^{\frac N 2}} \, e^{-\frac{r(x,x_0)^2}{4t}-c_3r(x,x_0)-c_4 t} \qquad \text{for all } t>0 \text{ and all }  x\in \hn \, ,
\end{align}
where $c_3,c_4,c_5>0$ only depend on $N,k$. In particular, from \eqref{6} we easily infer that
\begin{align}\label{7}
 \int_1^{+\infty}  \frac{ k_{\hn}(t,x,x_0)}{t^{1-s}} \,{\rm{d}t} \ge  \int_2^{+\infty}  \frac{ k_{\hn}(t,x,x_0)}{t^{1-s}} \,{\rm{d}t} \geq C\, e^{-\frac{r(x,x_0)^2}{8}-c_3r(x,x_0)}  \qquad \forall x\in \hn \, ,
\end{align}
for some $C=C(N,k,s)>0$. Moreover, from \eqref{gaussian} it follows that
\begin{equation}\label{5}
k_{\hn}(t,x,x_0)\le \frac C{t^\frac N2} \, e^{-\frac{r(x,x_0)^2}{5 t}} \qquad \text{for all } t>0 \text{ and all }  x\in \hn \, ,
\end{equation}
for some $ C=C(N,c)>0 $, where $ c $ is the constant appearing in \eqref{FK}. Since $r(x,x_0)^2 \geq \frac{5}{6} r(x,x_0)^2 + \frac{1}{6}$ if $r(x,x_0)\geq 1$, estimate \eqref{5} yields
\begin{align}\label{8}
 \int_0^{1}  \frac{k_{\hn}(t,x,x_0)}{t^{1-s}}  \,{\rm{d}t}\leq C\, e^{-\frac{r(x,x_0)^2}{6}} \int_0^{1} \frac{e^{-\frac{1}{30t}}}{t^{\frac{N}{2}+1-s}} \,{\rm{d}t}=C\, e^{-\frac{r(x,x_0)^2}{6}} \qquad \text{for all $x \in \hn$: }  r(x,x_0)\geq 1 \, ,
\end{align}
where now $C$ also depends on $ s $. Therefore, by combining \eqref{7} and \eqref{8}, we deduce that
$$ \int_0^{+\infty}  \frac{ k_{\hn}(t,x,x_0)}{t^{1-s}} \,{\rm{d}t}\leq C  \int_1^{+\infty}  \frac{ k_{\hn}(t,x,x_0)}{t^{1-s}} \,{\rm{d}t}$$
for another $C=C(N,k,c,s)>0$, which, plugged into \eqref{4}, gives (recall \eqref{fract})
\begin{align}\label{9}
 (-\Delta_{\hn})^{-s}  \psi (x) \geq  C \left\| \psi \right\|_1  \GM (x, x_0)    \qquad \text{for all $ x \in \hn $: }  r(x,x_0)\geq 1\,,
\end{align}
still for a suitable $C=C(N,k,c,s)>0$. Next we consider the case $r(x,x_0)< 1$. By combining \eqref{4} and \eqref{7} we immediately obtain
\begin{align}\label{10}
 (-\Delta_{\hn})^{-s}  \psi (x) \geq  C \left\| \psi \right\|_1   \qquad \text{for all $ x \in \hn $: }  r(x,x_0) < 1\,,
\end{align}
whereas, by virtue of \eqref{green-euc}, we have that
\begin{align}\label{11}
r(x,x_0)^{N-2s} \,  \GM (x, x_0) \leq C \qquad \forall x \in \hn \setminus \{ x_0 \} \, ,
\end{align}
where $C=C(N,k,c,s)>0$ is another constant as above. The lower bound in \eqref{boundgreen} is then a consequence of \eqref{9}, \eqref{10} and \eqref{11}.

Now we pass to the proof of the upper bound. Given any $x,y \in \hn $ and $ t \ge 1 $, still from Li-Yau estimates applied to $v(\cdot, \cdot)=k_{\hn}( \cdot +t-1,x,\cdot)$, $x_1=y$, $x_2=x_0$, $t_1=1$ and $t_2=2$, we have
$$
\begin{gathered}
k_{\hn}(t,x,y) \leq c_0 \, 2 \, ^{\beta} \, e^{c_1\sigma+c_2}\, k_{\hn}(t+1,x,x_0) = C \, k_{\hn}(t+1,x,x_0) \\
 \text{for all } t\geq 1 \text{ and all $x,y \in\hn$: }  r(y,x_0) \leq \sigma \, ,
 \end{gathered}
$$
where $C=C(N,k)>0$. We thus deduce that
\begin{equation}\label{12}
\begin{aligned}
\int_{B_{\sigma}(x_0)} \psi(y) \int_1^{+\infty}  \frac{ k_{\hn}(t,x,y)}{t^{1-s}} \,{\rm{d}t } \, \dg(y) \leq & \,  C \int_{B_{\sigma}(x_0)} \psi(y) \int_1^{+\infty}  \frac{ k_{\hn}(t+1,x,x_0)}{t^{1-s}} \,{\rm{d}t } \, \dg(y) \\
\leq & \,  C \left\| \psi \right\|_{\infty} \mu_M(B_{\sigma}(x_0))\,  \GM(x, x_0) \qquad \forall x \in \hn \, ,
\end{aligned}
\end{equation}
for another $C=C(N,k)>0$.

Let us assume first $r(x,x_0)\geq 2$; by \eqref{6} we have that
\begin{align}\label{r2}
k_{\hn}(t,x,y) \geq \frac{c_5}{t^{\frac N 2}} \, e^{-\frac{r(x,y)^2}{8}-c_3r(x,y)-c_4 t} \qquad \text{for all } t>2 \text{ and all }  x,y \in \hn \, .
\end{align}
From \eqref{r2}, we deduce that
\begin{align*}
\int_{B_{\sigma}(x_0)} \psi(y) \int_2^{+\infty} & \frac{ k_{\hn}(t,x,y)}{t^{1-s}} \,{\rm{d}t } \, \dg(y) \\
&\geq \left( \int_2^{+\infty}  \frac{ c_5 \,e^{-c_4 t}}{t^{\frac{N}{2}+1-s}} \,{\rm{d}t} \right) \left( \int_{B_\sigma(x_0)}  \,\psi(y)\,e^{-\frac{r(x,y)^2}{8}-c_3r(x,y)}\, \dg(y) \right)\\
&\geq \left( \int_2^{+\infty}  \frac{ c_5 \,e^{-c_4 t}}{t^{\frac{N}{2}+1-s}} \,{\rm{d}t} \right) e^{-6c_3^2} \left( \int_{B_\sigma(x_0)}  \,\psi(y)\,e^{-\frac{r(x,y)^2}{6}}\, \dg(y) \right)\\
&=C \left( \int_{B_\sigma(x_0)}  \,\psi(y)\,e^{-\frac{r(x,y)^2}{6}}\, \dg(y) \right) ,
\end{align*}
where $C=C(N,k,s)>0$. Furthermore, upon noticing that $r(x,x_0)\geq 2$ and $r(x_0,y)\leq \sigma< 1$ yield $r(x,y)> 1$, from \eqref{5} (with $ x_0 $ replaced by $y$) we get
\begin{align*}
k_{\hn}(t,x,y) \leq \frac C{t^\frac N2} \, e^{-\frac{r(x,y)^2}{5 t}}  \leq \frac C{t^\frac N2} \,  e^{-\frac{r(x,y)^2}{6}-\frac{1}{30t}} \qquad \text{for all } 0<t<1 \, ,
\end{align*}
for every such $ x $ and $ y $. As a result, we have
\begin{align*}
\int_{B_{\sigma}(x_0)} \psi(y) \int_0^{1} & \frac{ k_{\hn}(t,x,y)}{t^{1-s}} \,{\rm{d}t } \, \dg(y) \\
&\leq \left( \int_0^{1}  \frac{ C \,e^{-\frac{1}{30t}}}{t^{\frac{N}{2}+1-s}} \,{\rm{d}t} \right) \left( \int_{B_\sigma(x_0)}  \,\psi(y)\,e^{-\frac{r(x,y)^2}{6}}\, \dg(y) \right)\\
&=C \left( \int_{B_\sigma(x_0)}  \,\psi(y)\,e^{-\frac{r(x,y)^2}{6}}\, \dg(y) \right) \qquad \text{for all $ x \in \hn $: }  r(x,x_0) \ge 2 \, ,
\end{align*}
for another $C=C(N,k,c,s)>0$. Therefore, by combining the above estimates, we deduce in particular that
$$
\int_{B_{\sigma}(x_0)} \psi(y) \int_0^{+\infty}  \frac{ k_{\hn}(t,x,y)}{t^{1-s}} \,{\rm{d}t } \, \dg(y)  \le  C \,  \int_{B_{\sigma}(x_0)} \psi(y) \int_1^{+\infty}  \frac{ k_{\hn}(t,x,y)}{t^{1-s}} \,{\rm{d}t } \, \dg(y) \, ,
$$
whence, by virtue of \eqref{12},
\begin{align}\label{rlarge}
 (-\Delta_{\hn})^{-s}  \psi  (x) \leq C \left\| \psi \right\|_\infty \sigma^N \, \GM(x, x_0) \qquad \text{for all $ x \in \hn $: }  r(x,x_0) \ge 2 \, ,
\end{align}
where, by exploiting the curvature bound \eqref{ric}, we have used \eqref{uppervolume}, and $C>0$ is a possibly different constant still depending on $ N,k,c,s $ only.

In order to handle the case $r(x,x_0)< 2$, we observe that, as a consequence of \eqref{fract-bis} and \eqref{green-euc}, it holds
\begin{equation}\label{potpot}
(-\Delta_{\hn})^{-s} \psi  (x) \leq C \left\| \psi \right\|_\infty \int_{B_\sigma(x_0)} \, \frac{1}{r(x,y)^{N-2s}} \, \dg(y) \qquad \forall x \in \hn \, .
\end{equation}
Then, to complete the proof of the upper bound in \eqref{boundgreen}, we show that
\begin{equation}\label{int}
\int_{B_\sigma(x_0)} \frac{1}{r(x,y)^{N-2s}} \, \dg(y) \leq C \, \frac{\sigma^N}{r(x,x_0)^{N-2s}} \qquad \text{for all $ x \in \hn $: }  r(x,x_0) < 2 \, ,
\end{equation}
for a suitable $C=C(N,k,s)>0$. To this aim, let us first assume that $2 > r(x,x_0)\ge2\sigma$; the triangle inequality then yields
\[
r(x,y) \ge r(x,x_0) - \sigma\ge \frac 12 \, r(x,x_0) \qquad \text{for all }  y\in B_\sigma(x_0)\,.
\]
Therefore, in this case we have
\[
\int_{B_\sigma(x_0)} \frac{1}{r(x,y)^{N-2s}}\, \dg(y)\ \le 2^{N-2s} \, C\,\frac{\sigma^N}{r(x,x_0)^{N-2s}} \, ,
\]
where again we have used \eqref{uppervolume}, and $C$ is the positive constant appearing therein. If $r(x,x_0)<2\sigma$, we need to use some more delicate estimates. Let us denote by $\mathcal{H}_M^{N-1}$ the $(N-1)$-dimensional Hausdorff measure on $\hn$, and observe that $\mathcal{H}_M^{N-1}(\partial B_r(x))\le C r^{N-1}$ for all $r\in(0,3)$, all $ x \in \hn $ and a suitable $ C=C(N,k)>0 $ (this is a consequence of Bishop-Gromov -- see also the proof of Proposition \ref{smoothPhi-lemma} below). From $r(x,x_0)<2\sigma$ it follows that $B_\sigma(x_0)\subset B_{3\sigma}(x)$, hence, using the coarea formula, we get:
\begin{equation}\label{pot}\begin{aligned}
\int_{B_\sigma(x_0)} \frac{1}{r(x,y)^{N-2s}}  \, \dg(y) & \le \int_{B_{3\sigma}(x)}  \frac{1}{r(x,y)^{N-2s}}\, \dg(y)\\
&=\int_0^{3\sigma} \frac1{r^{N-2s}} \,\mathcal{H}_M^{N-1}(\partial B_r(x))\,{\textrm d}r\\&\le C \, \frac{3^{2 s}}{2 s} \, \sigma^{2s}
\leq C \, \frac{3^{2s} \, 2^{N-2s}}{2 s} \, \frac{\sigma^N}{r(x,x_0)^{N-2s}} \, .
\end{aligned}\end{equation}
Inequality \eqref{int} is therefore established. We finally prove the reverse bound
\begin{equation}\label{green-est-low}
\GM(x, x_0) \ge \frac{C}{r(x,x_0)^{N-2s}} \qquad \text{for all $ x \in \hn $: }  r(x,x_0) < 2 \, ,
\end{equation}
for some $C=C(N,k,s)>0$. To this end, it is convenient to exploit \eqref{6} along with the change of variable $t/r(x,x_0)^2=\tau$ (let $ x \neq x_0 $), as follows:
\[\begin{aligned}
\GM(x, x_0)&\ge c_5 \int_0^{+\infty} \frac{1}{t^{\frac N2+1-s}} \, e^{-\frac{r(x,x_0)^2}{4t}-c_3r(x,x_0)-c_4 t}  \, {\textrm d}t\\
&=\frac{c_5}{r(x,x_0)^{N-2s}}\int_0^{+\infty} \frac{1}{\tau^{\frac N2+1-s}} \, e^{-\frac{1}{4\tau}-c_3r(x,x_0)-c_4 \tau  r(x,x_0)^2}\,{\textrm d}\tau\\
& \geq \frac{c_5}{r(x,x_0)^{N-2s}}\int_0^{+\infty} \frac{1}{\tau^{\frac N2+1-s}} \, e^{-\frac{1}{4\tau}-2c_3-4c_4 \tau}\,{\textrm d}\tau\\
& = \frac{C}{r(x,x_0)^{N-2s}} \qquad \text{for all $ x \in \hn $: }  r(x,x_0) < 2 \, .
\end{aligned}
\]
By \eqref{potpot}, \eqref{int} and \eqref{green-est-low}, we may assert that
\begin{align*}
 (-\Delta_{\hn})^{-s}  \psi  (x) \leq C \left\| \psi \right\|_\infty \sigma^N \, \GM(x, x_0) \qquad \text{for all $ x \in \hn $: }  r(x,x_0) < 2 \, ,
\end{align*}
which, combined with \eqref{rlarge}, finally yields the upper bound in \eqref{boundgreen} and concludes the proof.
 \end{proof}

In the sequel we will also need a further technical result, which can easily be shown by using some of the estimates provided in the proof of Lemma \ref{lem.Phi.estimates}.

 \begin{lem}\label{potentials}
 Let $M$ satisfy Assumption \ref{general}, and let $x_0\in M$. There exist a nonnegative and nontrivial function $ \psi \in \mbox{L}_c^{\infty}(\hn)$ and two constants $c_1, c_2>0$, depending only on $ N,s,k,c  $ (in particular independent of $ x_0 $), such that for all nonnegative $f\in L^1_{\GM}(\hn)$ one has
\begin{equation}\label{quasi-pot}
c_1 \, \int_M f \, (- \Delta_{\hn})^{-s} \psi \,  {\rm d}\mu_M\le\left\| f\right\|_{L^1_{x_0,\GM}} \le c_2 \, \int_M f \, (- \Delta_{\hn})^{-s} \psi \,  {\rm d}\mu_M \, .
\end{equation}
 \end{lem}
 \begin{proof}
Given the results of Lemma \ref{lem.Phi.estimates}, it is immediate to note that there exist $c_1, c_2>0$ as in the statement such that
\[
\begin{aligned}
c_1 \, \int_{\hn\setminus B_1(x_0)} f \, (- \Delta_{\hn})^{-s}\psi \, {\rm d}\mu_M&\le \int_{\hn\setminus B_1(x_0)} f \, \GM(\cdot ,x_0) \, \dg \\ &\le c_2 \, \int_{\hn\setminus B_1(x_0)} f \, (- \Delta_{\hn})^{-s} \psi \, {\rm d}\mu_M \, ,
\end{aligned}
\]
provided, for instance, one picks $ \psi = \chi_{B_{1/2}}(x_0) $. Clearly, in this case, $ \| \psi \|_\infty =1 $, whereas $ \| \psi \|_1 \ge C=C(N,c) > 0 $ thanks to the non-collapse bound \eqref{non-coll}. Besides, one sees from \eqref{10}, \eqref{potpot} and \eqref{pot} that there exist another two constants $c_1, c_2>0$ as in the statement such that
\[
c_1 \, \int_{B_1(x_0)} f \, (- \Delta_{\hn})^{-s} \psi \, {\rm d} \mu_M \le \int_{B_1(x_0)} f \, \dg \le c_2 \, \int_{B_1(x_0)} f \, (- \Delta_{\hn})^{-s} \psi \, {\rm d} \mu_M \, ,
\]
with the same $\psi $ chosen above. Finally, the two-sided bound \eqref{quasi-pot} follows by summing up the above estimates, upon recalling the definition of   the   $ L^1_{x_0,\GM}(\hn) $ norm given in \eqref{normLGx0}.
 \end{proof}

 \begin{rem}\label{generalR}
For later purposes, we point out that the upper bound in \eqref{boundgreen} can be extended to all $\sigma \ge 1$, up to letting the constant $ \overline{C} $ depend on $ \sigma $ as well. Indeed, first of all one notices that now the constant $C $ in \eqref{12} also depends on $ \sigma $ through Li-Yau estimates. Then it is enough to observe that all the inequalities below \eqref{r2}, that lead to \eqref{rlarge}, continue to hold by replacing $2$ (as a bound over $ r(x,x_0) $) with  $1+\sigma$. Estimate \eqref{rlarge} follows, recalling that also \eqref{uppervolume} holds for all $ 0 < r < \sigma $ and some uniform multiplying constant dependent on $\sigma$ as well. Moreover, since $\mathcal{H}_M^{N-1}(\partial B_r(x)) \le C r^{N-1}$ for all $r< 2\sigma+1$ and all $ x \in \hn $, with another constant $C>0$ as above that in addition depends on $\sigma$, the coarea formula still gives \eqref{pot} subject to $r(x,x_0) < \sigma+1 $. The same holds for \eqref{green-est-low}, again with $2$ replaced by $ \sigma+1 $.
\end{rem}

\section{On the class of initial data}\label{init-data}

We provide some explicit examples, first on specific space forms and then on general manifolds that fit into our setting, of functions belonging to $L^1_{x_0,\GM}(\hn)$ and to $L^1_{\GM}(\hn)$, showing in particular that the latter form strictly wider classes than $ L^1(M) $.

\subsection{Admissible decay for initial data in the space forms $  {M=\rn}$ and $ {M=\mathbb{H}^{N}} $} \label{examples}
 In order to highlight the admissible decay rate for the kind of initial data we deal with, we provide some sufficient conditions for a function to belong to $L^1_{\GM}(\hn)$.
\begin{prop}\label{decay}
Let either $M=\rn$ or $M=\mathbb{H}^{N}$, and let $ u_0 \in L^{\infty}(M)$. Then, sufficient conditions for $u_0$ to belong to $L^1_{\GM}(\hn)$ are the following:
\begin{itemize}
\item $M=\rn$ and $\left| u_0(x) \right| \leq \dfrac{C}{|x|^a}$ for all $|x|\geq R$, for some $C, R>0$ and $a>2s$;
\item $M=\mathbb{H}^{N}$ and $\left|u_0(x)\right|\leq  \dfrac{C}{(r(x,o))^a}$ for all $r(x,o) \geq R$, for some $o\in M$, $C, R>0$ and $a>s$.
\end{itemize}
\end{prop}

It is worth mentioning that, in both cases, initial data are allowed to decay considerably \emph{slower} than functions in $L^1(M)$; this fact is more evident in the non-flat case where functions in $L^1(\mathbb{H}^{N})$ are expected to decay faster than $e^{-r(x,o)(N-1)}$. For proving the above statement it is clearly enough to show that the special functions
\begin{equation*}\label{ua}
u_{a}(x):=
 \left \{ \begin{array}{ll}
   1 & \mbox{ for } x\in \rn: \, |x|\leq 1\,,  \\
 \dfrac{1}{|x|^a}  &  \mbox{ for } x \in \rn: \, |x|> 1 \, ,
\end{array}
\right. \quad \text{with } a>2s \, ,
\end{equation*}
and
\begin{equation*}\label{wa}
w_{a}(x):=
 \left \{ \begin{array}{ll}
   1 & \mbox{ for } x\in\mathbb{H}^{N}: \, r(x,o)\leq 1\, ,  \\
 \dfrac{1}{(r(x,o))^a}  &  \mbox{ for } x\in\mathbb{H}^{N}: \, r(x,o)> 1 \, ,
\end{array}
\right. \quad \text{with } a>s \text{ and } o \in \mathbb{H}^{N} \text{ fixed}\, ,
\end{equation*}
satisfy $u_a \in L^1_{\GR}(\rn)$ and $w_a \in L^1_{\GH}(\mathbb{H}^{N})$. Since $u_a \not\in L^1(\rn)$ (for $ a \le N $) and $w_a \not\in L^1(\mathbb{H}^{N})$, these examples also prove the strict inclusion $ L^1(\hn) \subsetneq L^1_{\GM}(\hn) $ for $M=\rn$ and $M=\mathbb{H}^{N}$ (see Subsection \ref{examples-bis} for more about such inclusions in a general setting).

\begin{proof}[Proof of Proposition \ref{decay}]
Let $M=\rn$. We aim at showing that $u_a \in L^1_{\mathbb{G}_{\R^N}^s}(\rn)$. Since
\begin{align*}
\sup_{x_0\in \rn}  \int_{B_1(x_0)} |u_a(x)| \, \dx \leq \omega_N< +\infty\,,
\end{align*}
where $\omega_N:= \mu_{\rn}(B_1(0))$, we only need to prove that
\begin{align}\label{uak}
\sup_{x_0\in \rn} \int_{\rn \setminus B_1(x_0)} |u_a(x)| \,  \mathbb{G}_{\R^N}^s(x,x_0) \,\dx < +\infty  \, .
\end{align}
To this end, we can write
\begin{equation}\label{k99}
\begin{aligned}
\int_{\rn \setminus B_1(x_0)} |u_a(x)| & \,  \mathbb{G}_{\R^N}^s(x,x_0)\,\dx  \\
&\leq  \, \int_{(\rnn \setminus B_1(x_0)) \cap B_1(0)} C \,\dx +  \int_{(\rn \setminus B_1(x_0)) \cap B_1^c(0)} |u_a(x)| \, \mathbb{G}_{\R^N}^s(x,x_0) \, \dx \\
&\leq \,  C \, \omega_N + \int_{(\rn \setminus B_1(x_0)) \cap B_1^c(0)} |u_a(x)| \, \mathbb{G}_{\R^N}^s(x,x_0) \, \dx \\
& \leq \,  C \, \omega_N + C \,  \int_{\rn} \frac{1}{\left(| x | +1 \right)^a}  \, \frac{1}{\left(| x-x_0 |+1\right)^{N-2s}} \, \dx \, ,
\end{aligned}
\end{equation}
for some $C=C(N,s,a)>0$. If   $ x \in B_{| x_0 |/2}(0) ,$   we have $ | x-x_0 | \ge | x_0 |/2 $, thus
$$
\begin{aligned}
\int_{  B_{| x_0 |/2}(0)   } \frac{1}{\left(| x | +1 \right)^a}  \, \frac{1}{\left(| x-x_0 |+1\right)^{N-2s}} \, \dx \le & \, \frac{C}{\left(| x_0 |+1\right)^{N-2s}} \int_{B_{| x_0 |/2}(0) } \frac{1}{\left(| x | +1 \right)^a} \, \dx \\
 \le & \, C  \, \frac{\left(| x_0 |+1\right)^{N-a}}{\left(| x_0 |+1\right)^{N-2s}} \le C \, ,
\end{aligned}
$$
since $ a >2s $ (and we can assume w.l.o.g.~that $ a < N $). On the other hand, if $ x \in B_{2| x_0 |}^c(0) $, we have $ | x-x_0 | \ge | x |/2 $, thus
$$
\int_{B_{2| x_0 |}^c(0) } \frac{1}{\left(| x | +1 \right)^a}  \, \frac{1}{\left(| x-x_0 |+1\right)^{N-2s}} \, \dx \le \int_{\rn} \frac{1}{\left(| x | +1 \right)^{N+a-2s}}  \, \dx \le C \, .
$$
Finally, if  $ x \in B_{2| x_0 |}(0) \setminus B_{| x_0 |/2}(0) $, we have $ | x | \ge | x_0 |/2 $ and therefore
$$
\begin{aligned}
\int_{ B_{2| x_0 |}(0) \setminus B_{| x_0 |/2}(0)} & \frac{1}{\left(| x | +1 \right)^a}  \, \frac{1}{\left(| x-x_0 |+1\right)^{N-2s}} \, \dx \\
 \le & \,  \frac{C}{\left(| x_0 | +1 \right)^a} \int_{ B_{3| x_0 |}(x_0) }  \frac{1}{\left(| x-x_0 |+1\right)^{N-2s}} \, \dx \le C \, \frac{\left(| x_0 | +1 \right)^{2s}}{\left(| x_0 | +1 \right)^a} \le C \, .
\end{aligned}
$$
By combining the last three estimates and \eqref{k99}, we finally obtain \eqref{uak}.

Now let $ M=\hnn $. In order to show that $w_a \in L^1_{\GH}(\mathbb{H}^{N})$, at first we notice again that
\begin{align*}
\sup_{x_0\in \hnn}  \int_{B_1(x_0)} |w_a| \, \dgh \leq \mu_{\hnn}(B_1(x_0))=C < +\infty \, ,
\end{align*}
since the volume measure on $ \hnn $ is invariant w.r.t.~any pole. Then we only need to prove that
\begin{align}\label{wa-k}
\sup_{x_0\in \hnn}  \int_{\hnn \setminus B_1(x_0)} |w_a(x)|\,  \GH(x,x_0) \, \dgh(x) < +\infty  \, .
\end{align}
To this aim, recall that \eqref{Hyp.Green.HN0b} with $ \c=1 $ reads
\begin{align}\label{Gestimate}
 \GH(x,x_0) \le C \, \frac{ e^{-(N-1) r(x, x_0)}}{r(x, x_0)^{1-s}} \qquad \text{for all }(x,x_0)\in \mathbb{H}^{N} : \, r(x, x_0) \geq 1 \, ,
 \end{align}
for some $C=C(N,s)>0$. In particular, this yields
\begin{align*}
\int_{\hnn \setminus B_1(x_0)} |w_a(x)| & \,  \GH(x,x_0)\,\dgh(x) \\
&\leq  \int_{(\hnn \setminus B_1(x_0)) \cap B_1(o)} C\, \dgh +  \int_{(\hnn \setminus B_1(x_0)) \cap B_1^c(o)} |w_a(x)| \,  \GH(x,x_0)\,\dgh(x) \\
&\leq C \, \mu_{\hnn}(B_1(o)) +  \int_{(\hnn \setminus B_1(x_0)) \cap B_1^c(o)} |w_a(x)| \, \GH(x,x_0)\,\dgh(x) \, .
\end{align*}
In order to estimate the last integral, it is not restrictive to assume $1>a>s$. When $r(x_0,o) \ge  1$, by passing to polar coordinates centered at $x_0$, using the triangle inequality and recalling \eqref{Gestimate}, we have:
\begin{align*}
\int_{(\hnn \setminus B_1(x_0)) \cap B_1^c(o)}  |w_a(x)| & \,  \GH(x,x_0)\,\dgh(x)
\\
& \leq C \left[  \int_{1}^{r(x_0,o)}\frac{1}{(r(x_0,o)-r)^{a}} \,  \frac{1}{r^{1-s}} \, \mathrm{d} r  +  \int_{r(x_0,o)}^{+\infty} \frac{1}{(r-r(x_0,o))^{a}} \,  \frac{1}{r^{1-s}} \, \mathrm{d} r  \right]\\
& \leq  \frac{C}{(r(x_0,o))^{a-s}}  \left[  \int_{0}^{1}\frac{1}{(1-\rho)^{a}} \,  \frac{1}{\rho^{1-s}} \,  \mathrm{d} \rho +  \int_{1}^{+\infty}  \, \frac{1}{(\rho-1)^{a}} \, \frac{1}{\rho^{1-s}} \, \mathrm{d} \rho \right]\leq C \, ,
\end{align*}
where the latter positive constant only depends on $N,s,a$. Similarly, if $r(x_0,o) < 1 $, we get
$$
\int_{(\hnn \setminus B_1(x_0)) \cap B_1^c(o)}  |w_a(x)| \, \GH(x,x_0) \, \dgh(x) \leq  C \, \int_{1}^{+\infty} \frac{1}{(r-1)^{a}} \,  \frac{1}{r^{1-s}} \, \mathrm{d}r \le C \, ,
$$
for another $ C=C(N,s,a)>0 $ as above. Therefore, in both cases \eqref{wa-k} follows.
\end{proof}

\subsection{The class $L^1_{\GM}(\hn)$ is strictly included between $ L^1(M) $ and $L^1_{x_0, \GM }(\hn)$}\label{examples-bis}

  As we have observed in Subsection \ref{def-wds}, under Assumption \ref{general} the inclusions $ L^1(M) \subseteq L^1_{\GM}(\hn) \subseteq L^1_{x_0,\GM }(\hn) $ trivially hold. By means of an explicit construction, we can prove that in fact they are strict.
\begin{prop} 
Let $M$ satisfy Assumption \ref{general}. Then, for all $x_0\in \hn$, we have
$$ L^1(\hn) \subsetneq L^1_{\GM}(\hn) \subsetneq L^1_{x_0,\GM}(\hn)\,.$$
\end{prop}
\begin{proof}
We first show that $ L^1(\hn) \subsetneq L^1_{\GM}(\hn) $. To this aim, consider the following function:
\begin{equation}\label{ee-ff}
\hat{u} := \sum_{j=1}^\infty \chi_{B_1(o_j)} \, ,
\end{equation}
where $ \{ o_j \} \subset \hn $ is any sequence of points such that (let $ o \in \hn $ be fixed)
$$
r(o_j,o) = e^{j} \qquad \forall j \in \mathbb{N} \, .
$$
In view of the non-collapse bound \eqref{non-coll} it is apparent that $ \hat{u} \not\in L^1(\hn) $, since (note that the balls $ \{ B_1(o_j) \} $ are disjoint)
$$
\left\| \hat{u} \right\|_{L^1(\hn)} =  \sum_{j=1}^\infty \mu_\hn(B_1(o_j)) \ge \sum_{j=1}^\infty C = + \infty \, .
$$
Here, and in the sequel, we let $ C>0 $ denote a generic constant that only depends on $ N , k, s $ and the constant $c$ appearing in \eqref{FK}, whose precise value may actually change from line to line. In order to prove that $ \hat{u} \in L^1_{\GM}(\hn) $, first of all let us observe that for every $ j \in \mathbb{N} $ and $ x_0 \in \hn $ it holds
\begin{equation}\label{example-0}
\begin{aligned}
\int_{\hn \setminus B_1(x_0)}  \chi_{B_1(o_j)}(x) \,  \GM(x,x_0) \, \dg(x) = & \, \int_{B_1(o_j) \cap (\hn \setminus B_1(x_0))}  \GM(x,x_0) \, \dg(x) \\
 \le & \, C \int_{B_1(o_j) \cap (\hn \setminus B_1(x_0))}  \frac{1}{r(x,x_0)^{N-2s}} \, \dg(x) \\
 \le & \, C \int_{B_1(o_j) \cap (\hn \setminus B_1(x_0))}  \frac{1}{\left[(r(o_j,x_0)-1)\vee 1\right]^{N-2s}} \, \dg(x)  \\
 \le & \, C \, \frac{\mu_\hn(B_1(o_j))}{\left[ r(o_j,x_0) \vee 1 \right]^{N-2s}} \le \frac{C}{\left[ r(o_j,x_0) \vee 1 \right]^{N-2s}} \, ,
\end{aligned}
\end{equation}
where we took advantage of \eqref{green-euc} along with the triangle inequality and the fact that $ \mu_\hn(B_1(o_j)) \le C $, consequence of \eqref{uppervolume}.  As a result, we deduce that
\begin{equation}\label{example-1}
\begin{aligned}
\int_{\hn \setminus B_1(x_0)}  \left| \hat{u}(x) \right|  \GM(x,x_0) \, \dg(x) = & \, \sum_{j=1}^\infty \int_{\hn \setminus B_1(x_0)}  \chi_{B_1(o_j)}(x) \,  \GM(x,x_0) \, \dg(x) \\
\le & \, C \, \sum_{j=1}^\infty \frac{1}{\left[ r(o_j,x_0) \vee 1 \right]^{N-2s}} \, .
\end{aligned}
\end{equation}
Clearly, either $ x_0 \in B_{e^2}(o) $ or there exists a unique $ j_0 \in \mathbb{N} $, with $ j_0 \ge 2 $, such that $ x_0 \in \overline{B}_{e^{j_0+1}}(o) \setminus B_{e^{j_0}}(o) $. In the former case, still by means of the triangle inequality, we have
\begin{equation}\label{example-2}
\sum_{j=1}^\infty \frac{1}{\left[ r(o_j,x_0) \vee 1 \right]^{N-2s}} \le \sum_{j=1}^\infty \frac{1}{\left[ \left( r(o_j,o)-e^2 \right) \vee 1  \right]^{N-2s}} \le C \, \sum_{j=1}^\infty \frac{1}{r(o_j,o)^{N-2s}} =  C \, \sum_{j=1}^\infty \frac{1}{e^{(N-2s)j}} =C \, .
\end{equation}
 The latter case requires a little more computations. Firstly, we can split the last series in \eqref{example-1} as follows:
\begin{equation}\label{example-3}
\sum_{j=1}^\infty \frac{1}{\left[ r(o_j,x_0) \vee 1 \right]^{N-2s}} = \sum_{j=1}^{j_0} \frac{1}{\left[ r(o_j,x_0) \vee 1 \right]^{N-2s}}  + \sum_{j=j_0+1}^\infty \frac{1}{\left[ r(o_j,x_0) \vee 1 \right]^{N-2s}} \, .
\end{equation}
On the one hand, from the triangle inequality we infer that
$$
r(o_j,x_0) \ge r(o_j,o) - r(o,x_0) \ge  e^{j} - e^{j_0+1} = e^{j}\left(1-e^{j_0+1-j} \right) \ge e^{j} \, \frac{e-1}{e} \qquad \forall j \ge j_0+2 \, ,
$$
 whence
\begin{equation}\label{example-4}
 \sum_{j=j_0+1}^\infty \frac{1}{\left[ r(o_j,x_0) \vee 1 \right]^{N-2s}} \le 1 + \sum_{j=j_0+2}^\infty \frac{1}{ r(o_j,x_0)^{N-2s}} \le 1 + C \sum_{j=j_0+2}^\infty \frac{1}{ e^{(N-2s)j}} \le  C \, .
 \end{equation}
On the other hand, still the triangle inequality entails
$$
\begin{gathered}
r(o_j,x_0) \ge r(o,x_0)-r(o_j,o) \ge e^{j_0} - e^{j} = e^{j_0-j} \, e^{j} \left( 1 - e^{-j_0+j} \right) \ge e^{j_0-j} \, (e-1) \\ \forall j=1,\ldots,j_0-1 \, ,
\end{gathered}
$$
so that
\begin{equation}\label{example-5}
\sum_{j=1}^{j_0} \frac{1}{\left[ r(o_j,x_0) \vee 1 \right]^{N-2s}}  \le 1 + \sum_{j=1}^{j_0-1} \frac{1}{r(o_j,x_0)^{N-2s}} \le 1 + C \, \sum_{j=1}^{j_0-1} \frac{1}{e^{(N-2s)(j_0-j)}} \le C \, .
\end{equation}
By combining \eqref{example-1}, \eqref{example-2}, \eqref{example-3}, \eqref{example-4} and \eqref{example-5}, we finally obtain the bound
$$
\int_{\hn \setminus B_1(x_0)} \left|\hat{u}(x) \right| \GM(x,x_0) \, \dg(x) \le  C \qquad \forall x_0\in \hn \, ,
$$
while $ \int_{B_1(x_0)} \left| \hat{u} \right| \dg \le C $ since $\| \hat{u} \|_{L^\infty(\hn)} \le 1  $ and again $ \mu_\hn(B_1(x_0)) \le C $. We can therefore assert that $ \hat{u} \in L^1_{\GM}(\hn) $.

In order to show that also $ L^1_{\GM}(\hn) \subsetneq L^1_{x_0,\GM}(\hn) $, we can slightly modify the above construction as follows:
$$
\tilde{u} := \sum_{j=1}^\infty j \,  \chi_{B_1(o_j)} \, .
$$
Such a function does not belong to $ L^1_{\GM}(\hn) $, since
$$
\begin{aligned}
\left\| \tilde{u} \right\|_{L^1_{\GM}} \ge \sup_{x_0 \in M}  \int_{B_1(x_0)} \sum_{j=1}^\infty j \,  \chi_{B_1(o_j)} \, \dg = & \, \sup_{x_0 \in M} \sum_{j=1}^\infty j \, \mu_\hn(B_1(x_0) \cap B_1(o_j))
 \\
 \ge & \, \sup_{i \ge 1} \sum_{j=1}^\infty j \, \mu_\hn(B_1(o_i) \cap B_1(o_j)) \\
 = & \, \sup_{i \ge 1} i \, \mu_\hn(B_1(o_i)) = +\infty \, ,
\end{aligned}
$$
recalling again the non-collapse bound \eqref{non-coll}. On the other hand, for an arbitrary (but fixed) $x_0 \in \hn$, we have:
$$
\begin{aligned}
\left\| \tilde{u} \right\|_{L^1_{x_0,\GM}} = & \, \sum_{j=1}^\infty \int_{B_1(x_0)} j \,  \chi_{B_1(o_j)} \,\dg   + \sum_{j=1}^\infty \int_{\hn \setminus B_1(x_0)}  j \,  \chi_{B_1(o_j)}(x) \, \GM(x,x_0)\,\dg(x)   \\
\le & \, \sum_{j=1}^\infty j \, \mu_\hn(B_1(x_0) \cap B_1(o_j)) +  \sum_{j=1}^\infty j \, \frac{C}{\left[ r(o_j,x_0) \vee 1 \right]^{N-2s}} \\
\le & \, C \, \log(r(x_0,o)+1) +  \sum_{j=1}^\infty j \, \frac{C}{\left[ (r(o_j,o)-r(x_0,o)) \vee 1 \right]^{N-2s}} \\
\le & \, C \, \log(r(x_0,o)+1)  + C_{x_0}  \sum_{j=1}^\infty \frac{j}{e^{(N-2s)j}} < +\infty \, ,
\end{aligned}
$$
 where we have exploited \eqref{example-0} and the fact that $ \mu_\hn(B_1(o_j)) \le C $. Note that here we let $ C_{x_0}>0 $ denote a generic constant as above, which in addition may depend on $x_0$. As a result, we deduce that $ \tilde{u} \in L^1_{x_0,\GM}(\hn) $ for every $ x_0 \in \hn $.
\end{proof}

\begin{rem}\label{no-monotone-app}
Given $ f \in  L^1_{\GM}(\hn)  $, with $ f \ge 0 $, as a consequence of the monotone convergence theorem it is plain that if $ \{ f_n \} \subset L^1(M) \cap L^\infty(M) $ is an increasing sequence of nonnegative functions that converges pointwise to $ f $, then
$$
\left\| f-f_n \right\|_{L^1_{x_0,\GM}} \underset{n \to \infty}{\longrightarrow} 0 \qquad \forall x_0 \in \hn \, .
$$
However, in general, convergence w.r.t.~the norm $ \| \cdot \|_{L^1_{\GM}} $ does not occur. Indeed, if $ \hat{u} $ is the same function as in \eqref{ee-ff} and we pick the monotone increasing sequence $ \hat{u}_n := \chi_{B_n(o)} \hat{u} $, we have
$$
\left\|  \hat{u} - \hat{u}_n \right\|_{L^1_{\GM}} \ge \sup_{x_0 \in \hn} \int_{B_1(x_0)} \left| \hat{u} - \hat{u}_n \right| \dg \ge \sup_{j \ge n+1} \int_{B_1(o_j)} \left| 1 - 0 \right| \dg \ge C > 0 \qquad \forall n \in \mathbb{N} \, .
$$
\end{rem}

\section{Existence and fundamental estimates for approximate solutions} \label{smooth}

In this section we prove existence of WDS to \eqref{NFDE} for initial data lying in the space $ L^1(\hn) \cap L^\infty(\hn) $, along with several additional key properties that will be crucial in the construction of WDS for general initial data as in the statement of Theorem \ref{thm-existence}.

The strategy is qualitatively standard and takes advantage of nonlinear semigroup theory. On the other hand, the fact that semigroup solutions are also Weak Dual Solutions is not obvious and will be proved explicitly.

\subsection{Nonlinear semigroup in ${L^1(\hn)}$: mild vs.~Weak Dual Solutions} \label{NS-abstract}
The abstract theory in $L^1(\hn)$ developed by B\'enilan, Br\'ezis, Crandall, Liggett, Pazy, Pierre (see \cite{BCr,CR,CP,V} and references therein) does apply to our setting. Indeed, we aim at solving an equation of the form $\partial_t u=-\mathcal{L}[\varphi(u)]$, where the operator $\mathcal{L}: {\rm Dom}(\mathcal{L})\subset L^1(\hn)\to L^1(\hn)$, namely the $s$-fractional Laplacian on $M$, is densely defined, linear, m-accretive and with an order-preserving resolvent, whereas the nonlinearity $\varphi(r)=r^m$ satisfies standard ``monotonicity conditions'' according to \cite{CP}.

 In fact, this general theory allows one to establish existence and uniqueness for a larger class of nonlinearities and operators, but to our purposes it is enough to stick to the choices $\mathcal{L}=(-\Delta_{\hn})^s$ and $\varphi(r)=r^m$. More precisely, such a linear operator and nonlinearity satisfy the assumptions of {\cite[Propositions 1 and 2]{CP}, which in particular imply that the corresponding nonlinear operator obtained by composition is also m-accretive (up to an approximation) and its resolvent is order preserving, a property sometimes called $T$-accretivity (see e.g.~\cite[Chapter 10]{V}). We shall justify these assertions in the proof of the next proposition.}

The semigroup solutions obtained by resorting to the above recalled theory, are typically referred to as \it mild solutions\rm, and can be constructed through a suitable \emph{discretization procedure} which gives rise to $L^1$-continuous curves. We summarize these results in the following proposition, and for simplicity (since it is  {more convenient}   to our aims) we only focus on nonnegative initial data in $ L^1(M) \cap L^\infty(M) $, although the theory is well suited to treat  {also}   $  L^1(M)$ data, possibly sign changing.

\begin{prop}\label{BCPP-Thm}
	Let $u_0,v_0\in L^1(\hn) \cap L^\infty(\hn) $, with $ u_0,v_0 \ge 0 $. Then there exist unique nonnegative mild solutions $u,v\in C^0([0,+\infty); L^1(\hn))$ to problem \eqref{NFDE} corresponding to the initial data $u_0,v_0$, respectively, such that
	\begin{equation}\label{comparison.L1}
	\int_{\hn} \left( u(t,x)-v(t,x)\right)_+ \dg(x) \le \int_{\hn} \left( u_0(x)-v_0(x)\right)_+ \dg(x) \qquad \forall t\ge 0\,.
	\end{equation}
	As a consequence,
	\begin{equation*}
	\left\| u(t)-v(t) \right\|_{L^1(\hn)}\le \left\| u_0-v_0 \right\|_{L^1(\hn)} \qquad \forall t\ge 0 \, .
	\end{equation*}
	Moreover, such nonnegative mild solutions enjoy the following time monotonicity property:
	\begin{equation}\label{mon-est.OLD}
	\mbox{the map} \quad t \mapsto t^{\frac{1}{m-1}} u(t,x) \quad \mbox{is (essentially) nondecreasing for a.e.~$x \in \hn \, .$}
	\end{equation}
	Finally, they satisfy the following $ L^p(M) $-nonexpansivity property:
	\begin{equation}\label{eq-wm-1}
	\left\| u(t) \right\|_{L^p(\hn)} \leq \left\| u_0 \right\|_{L^p(\hn)} \qquad \mbox{for all $t\ge 0$ and all $1 \le p \leq \infty \, .$}
	\end{equation}
\end{prop}
\begin{proof}
First of all, we observe that the operator $(-\Delta_M)^s$ is the generator of a strongly continuous contraction semigroup in $ L^1(M) $, namely the semigroup subordinated to the heat flow on $M$ via the Bernstein function $ \R^+ \ni z \mapsto z^s $ (see \cite[Definition 4.3.2 and Example 3.9.16]{J}). 
Since the latter is well known to be sub-Markovian and in particular order preserving, the same is true for the subordinated semigroup, see for instance \cite[Corollary 4.3.4]{J} (the fact that these results are stated in $ \mathbb{R}^N $ is irrelevant). Hence, we can deduce that $ (-\Delta_M)^s $ is m-accretive in $ L^1(M) $ and the 
inequality
\begin{equation}\label{s-v}
\int_M \beta(u) \,  (-\Delta_M)^s u \,  \dg \ge 0
\end{equation}
holds, essentially whenever the integral makes sense (see \cite[formula (A3)]{CP}), where $ \beta : \mathbb{R} \to \mathbb{R} $ is an arbitrary non-decreasing continuous function with $ \beta(0) = 0 $.

The claimed results can now be deduced from \cite{CP} and standard nonlinear semigroup theory for accretive operators (for which we refer to the monograph \cite[Chapter 10]{V} and references quoted), since $ (-\Delta_M)^s $ and the nonlinearity $u \mapsto u^m$ (to be interpreted as $ u|u|^{m-1} $ for sign-changing functions) comply with the conditions required there.

We now show how \eqref{eq-wm-1} can be obtained for the discrete-time solutions and then extended to the limit (mild) solution,   also in order to write down explicitly the  corresponding time-discretization argument (i.e.~an implicit Euler scheme), of which we will make use below. To this end, let $T>0$ and $ n \in \mathbb{N} $  be fixed and set
	$t_k := \frac{k}{n} \, T$, for any integer $ 0 \leq k \leq n$,
	with constant time step $h := t_{k+1} - t_{k} = \frac{T}{n}.$ The mild semigroup solution, starting from $ u_0 $, is a continuous curve $ u \in C^0([0,T];L^1(M)) $ obtained as   the uniform limit in $ L^1(M) $ (see \cite[Theorem 10.16]{V})   of the piecewise-constant curves
	$$
	u_n(t) := u_k \quad \text{if }  t_k < t \le t_{k+1} \, , \qquad u_n(0):=u_0 \, ,
	$$
	where each element $ u_k $ is recursively defined as the solution to the following fractional elliptic equation:
	\begin{equation}\label{eq-approximate}
	h \, (-\Delta_{\hn})^s \!\left(u_{k+1}\right)^m    +  u_{k+1} = u_{k}  \qquad \mbox{in} \ \hn \, ,
	\end{equation}
	whose solvability is guaranteed by the above recalled running assumptions on the operator and on the nonlinearity.   More precisely, thanks to \cite[Proposition 2]{CP}, the solution of \eqref{eq-approximate} is obtained as the $ L^1(M) $-limit as $ \varepsilon \to 0^+ $ of the solutions to the perturbed problems
		\begin{equation}\label{eq-approximate-bis}
	h \, (-\Delta_{\hn})^s \!\left(u_{k+1,\varepsilon}\right)^m + \varepsilon \, h \,   u_{k+1,\varepsilon}^{m} +  u_{k+1,\varepsilon} = u_{k}  \qquad \mbox{in} \ \hn \, ,
		\end{equation}
	which are always well defined because the nonlinear operator $ u \mapsto h \, (-\Delta_{\hn})^s \!\left(u^m\right)+ \varepsilon \, h \,   u^{m} $ is m-accretive for every fixed $ \varepsilon>0 $ (see \cite[Proposition 1]{CP}). If $ u_k \in L^1(\hn) \cap L^\infty(\hn) $, then from \eqref{s-v} we infer the $L^p(M) $-nonexpansivity property
	\begin{equation}\label{ke}
		\left\| u_{k+1,\varepsilon} \right\|_{L^p(\hn)} \leq \left\| u_k \right\|_{L^p(\hn)}  \qquad \text{for all } 1 \le p \le \infty \, .
	\end{equation}
	In order to show \eqref{ke}, it is enough to multiply \eqref{eq-approximate-bis} by $ u_{k+1,\varepsilon}^{p-1} $ (via a truncation), integrate in $ \hn $ and exploit \eqref{s-v}, along with an application of H\"older's inequality. Therefore, the convergence of $ \{ u_{k+1,\varepsilon} \}_\varepsilon $ to $ u_{k+1} $ actually occurs in every $ L^p(\hn) $, for $ 1 \le p < \infty $. In particular, we can infer that \eqref{eq-approximate} is satisfied in the classical sense, i.e.~$ u_{k+1}^m $ belongs to the $ L^2(\hn) $ domain of the self-adjoint operator $ (-\Delta_{\hn})^s $. Moreover, still from \eqref{ke}, by induction each $u_k$ does belong to $L^{1}(\hn) \cap L^{\infty}(\hn)$ and satisfies the nonexpansivity estimate

	$$
	\left\| u_{k+1} \right\|_{L^p(\hn)} \leq \left\| u_k \right\|_{L^p(\hn)} \leq \ldots \leq \left\| u_0 \right\|_{L^p(\hn)}
	$$
	for all $ 1 \le p \le \infty $, whence, due to the definition of $ u_n $,
	$$
	\left\| u_n(t) \right\|_{L^p(\hn)} \le \left\| u_0 \right\|_{L^p(\hn)}  \qquad \forall t \in [0,T] \, .
	$$
	We conclude the proof by letting $n \rightarrow \infty$, recalling that $ \{ u_n(t) \}_n $ converges strongly to $ u(t) $ in $ L^1(M) $ (uniformly in time on $ [0,T] $) and each $ L^p(M) $ norm is lower semicontinuous w.r.t.~such convergence, so as to obtain
	$$
 \left\| u(t) \right\|_{L^p(\hn)}   \le  \liminf_{n\to\infty} \left\| u_n(t) \right\|_{L^p(\hn)} \le \left\| u_0 \right\|_{L^p(\hn)}  \qquad \forall t \in [0,T] \, ,
 	$$
namely \eqref{eq-wm-1}.   The $ L^1(\hn) $-ordering estimate \eqref{comparison.L1} can be obtained in a similar way; in this case, one considers the discrete solutions to
$$
	h \, (-\Delta_{\hn})^s \!\left(v_{k+1}\right)^m    +  v_{k+1} = v_{k}  \qquad \mbox{in} \ \hn \, ,
$$
takes the difference with \eqref{eq-approximate} and multiplies by $ \operatorname{sign}_+(u_{k+1}-v_{k+1}) $ (via an approximation), integrating in $ \hn $ and using again \eqref{s-v}. This yields \eqref{comparison.L1} at the discrete level, and the passage to the limit as $ n \to \infty $ is performed exactly as above. Note that \eqref{comparison.L1} entails, in particular, the positivity of the solution $ u $ if $ u_0 \ge 0 $.

Finally, the monotonicity property \eqref{mon-est.OLD} is a direct consequence of \cite[Theorem 4]{CP}.
\end{proof}

Now that we have constructed mild solutions in $L^1(\hn) \cap L^\infty(M) $, we aim at showing that the latter are also \emph{Weak Dual Solutions} according to Definition \ref{defi_WDS}. We devote the rest of this subsection to proving such crucial fact, borrowing ideas from \cite[Section 7]{BV1}.

\begin{lem}\label{lemma-inv}
Let $ \hn $ satisfy Assumption \ref{general}, and let $ p \in \left( 1 , \frac{N}{2s} \right) $. Then the operator $ (-\Delta_\hn)^{-s} $ is continuous from $ L^p(\hn) $ to $ L^q(\hn) $ for $ q = \frac{pN}{N-2sp} $. Moreover, if $ u  , \left( -\Delta_\hn \right)^s u\in L^p(\hn) $, the following left-inverse formula holds:
\begin{equation}\label{left-infv}
\left( -\Delta_\hn \right)^{-s}\left[ \left( -\Delta_\hn \right)^{s} u \right] = u \, .
\end{equation}
\end{lem}
\begin{proof}
As concerns the $ L^p $-$L^q$ continuity property of the (linear) operator $ (-\Delta_\hn)^{-s} $, it is a direct consequence of \eqref{FK} via the heat-kernel bound \eqref{gaussian}, thanks to the results of \cite{Var} (see also \cite[Introduction]{CM} for a nice survey). In particular, if we let $ \{ T_t \}_{t \ge 0} $ denote the subordinated semigroup generated by $ \left(-\Delta_\hn \right)^s $, still as a consequence of \eqref{gaussian} it is not difficult to infer that
$$
\lim_{t \to +\infty} \left\| T_t u \right\|_{L^p(\hn)} = 0 \, ,
$$
whereas, in view of the assumptions on $ u $ and the definition of $ \{ T_t \}_{t \ge 0} $, it follows that
$$
\lim_{h \to 0^+} \frac{u - T_h u}{h} = \left( -\Delta_\hn \right)^s u \qquad \text{in } L^p(\hn) \, .
$$
Recalling that the operator $ (-\Delta_\hn)^{-s} $ introduced in \eqref{fract-bis} can equivalently be rewritten as
$$
(-\Delta_\hn)^{-s} = \int_0^{+\infty} T_t \, dt
$$
(by means of Fubini's theorem and subordination), in order to achieve \eqref{left-infv} one can reproduce verbatim the proof of
\cite[Proposition 6.2.12]{J-3}, taking advantage of the above properties.
\end{proof}

\begin{prop}\label{thm-weak-mild}
Let $ \hn $ satisfy Assumption \ref{general}. Let $u$ be the mild solution to \eqref{NFDE} corresponding to any nonnnegative initial datum $u_0 \in L^1(\hn) \cap L^\infty(M) $. Then, $u$ is a WDS in the sense of Definition~\ref{defi_WDS}.
\end{prop}

\begin{proof}
	We mainly follow ideas behind the proof of \cite[Proposition 7.2]{BV1}, which deals with FPME on bounded Euclidean domains, thus we will just emphasize the crucial points for convenience of the reader.
	
	Let us go back to the time discretization procedure introduced in \eqref{eq-approximate} (with $T>0$ arbitrary but fixed). Thanks to   Lemma \ref{lemma-inv}, the expressions $(-\Delta_\hn)^{-s}u_k$ and $ (-\Delta_\hn)^{-s}\left[ \left(-\Delta_\hn \right)^s \! \left(u_{k+1}\right)^m \right]	 $ make sense as functions in $ L^q(\hn) $ for all $k$, as from Proposition \ref{BCPP-Thm} we know in particular that $ u_k ,  \left(-\Delta_\hn \right)^s \! \left(u_{k+1}\right)^m \in L^1(\hn)\cap L^\infty(\hn)$.

Let us take an arbitrary test function $\psi \in C^1_c( (0, T); L_c^{\infty}(\hn))$ and set $ \psi_k := \psi(t_k,\cdot) $. We multiply \eqref{eq-approximate} by $ (-\Delta_\hn)^{-s}\psi_k $, integrate in $ M $, and sum all terms over $ k=0,\ldots,n-1 $ to get
	$$
\sum_{k =0}^{n-1} \int_{\hn}\left[h \, (-\Delta_{\hn})^s \!\left(u_{k+1}\right)^m    +  u_{k+1}\right](-\Delta_\hn)^{-s}\psi_k \, \dg = \sum_{k =0}^{n-1} \int_{\hn} u_{k}(-\Delta_\hn)^{-s}\psi_k\, \dg\ .
$$
By using Fubini's theorem and Lemma \ref{lemma-inv}, we then obtain:
$$	\sum_{k =0}^{n-1} \int_{\hn} \left[ (-\Delta_{\hn})^{-s} u_{k+1} - (- \Delta_{\hn})^{-s} u_k \right] \psi_{k} \, \dg
	= - h \, \sum_{k=0}^{n-1} \int_{\hn} \left(u_{k+1}\right)^m \psi_{k} \, \dg \, ,
$$
	which can be rewritten as
	\begin{align*}
	\int_{\hn} \, [\psi_{n-1} \, (-\Delta_{\hn})^{-s} u_n \, & - \, \psi_0 \, (-\Delta_{\hn})^{-s} u_0] \, \dg \\
	- \underbrace{\sum_{k =1}^{n-1} \int_{\hn} \left(\psi_k - \psi_{k-1}\right) (-\Delta_{\hn})^{-s} u_k \, \dg}_{=: I_1}
	& = - \underbrace{h \, \sum_{k =0}^{n-1} \int_{\hn} (u_{k+1})^m \, \psi_k \, \dg}_{=: I_2} .
	\end{align*}
	As concerns $ I_1 $, we have:
	\begin{align}
	I_1 & =  \sum_{k =1}^{n-1}  h \, \int_{\hn} \frac{\psi_k - \psi_{k-1}}{h} \, (-\Delta_{\hn})^{-s} u_k \, \dg \label{e012}  \\
	& = \sum_{k =1}^{n-1}  h \, \int_{\hn} \partial_t \psi ({t}_k,\cdot) \, (-\Delta_{\hn})^{-s} \!\left( u_k - \tilde{u}_k\right)  \dg
	+  \sum_{k =1}^{n-1}  h \, \int_{\hn} \partial_t \psi ({t}_k,\cdot) \, (-\Delta_{\hn})^{-s}  \tilde{u}_k \, \dg + \rho_n  \, ,\nonumber
	\end{align}
	where we set $  \tilde{u}_k := u(t_k ,\cdot)  $ and
	$$
	\rho_n :=  \sum_{k =1}^{n-1}  h \, \int_{\hn} \left[ \frac{\psi_k - \psi_{k-1}}{h} - \partial_t \psi ({t}_k,\cdot) \right] (-\Delta_{\hn})^{-s} u_k \, \dg \, .
	$$
	Since $ \psi $ belongs to $ C^1_c( (0, T); L_c^{\infty}(\hn))$ and $ \{ u_k \} $ is uniformly bounded in every $ L^p(M) $ space, the same holds for $  \{ (-\Delta_{\hn})^{-s} u_k \} $ in some $ L^q(M) $ space by virtue of the continuous injections recalled above,   thus   it is readily seen that $ \rho_n \to 0 $ as $ n \to \infty $. Let us then focus on the other two integrals on the r.h.s.~of \eqref{e012}. As for the first one,   recalling that by construction

	$$\sup_{t \in [0,T]} \left\| u_n(t) - u(t) \right\|_{L^1(\hn)}\underset{n \to \infty}{\longrightarrow} 0 $$
	  (recall the proof of Proposition \ref{BCPP-Thm}),   we have:
	\begin{align*}
 &\left| \sum_{k =1}^{n-1}  h \, \int_{\hn} \partial_t \psi ({t}_k,\cdot) \, (-\Delta_{\hn})^{-s} \!\left( u_k - \tilde{u}_k\right)  \dg \right|
	=  \left| \sum_{k =1}^{n-1}  h \, \int_{\hn}  (-\Delta_{\hn})^{-s} \!\left[ \partial_t \psi ({t}_k,\cdot) \right] \!\left( u_k - \tilde{u}_k\right)  \dg \right| \\
&\leq  \, \sum_{k =1}^{n-1}  h \left\|  (-\Delta_{\hn})^{-s} \!\left[ \partial_t \psi ({t}_k) \right] \right\|_{L^\infty(M)} \left\|  u_k - \tilde{u}_k \right\|_{L^1(M)}
	\leq  \, C  \sup_{t \in [0,T]} \left\| u_n(t) - u(t) \right\|_{L^1(\hn)}  \,
	\sum_{k =1}^{n-1}  h  \underset{n \to \infty}{\longrightarrow} 0 \, ,
	\end{align*}
	since $h=T/n$, where the constant $C>0$ depends only on $N , k,c ,  s $ the support of $ \psi  $ and the $ L^\infty(\hn) $ norm of $ \partial_t \psi $, in agreement with the global estimates provided in Remark \ref{generalR}. Note that in the first passage we have exploited Fubini's theorem, which is easily seen to be applicable still in view of such estimates. On the other hand, we recognize that the second integral is nothing but a Riemann sum of the (continuous, recall Remark \ref{remdef}) function
	$$
	t \mapsto \int_{\hn} \partial_t \psi (t,\cdot) \, (-\Delta_{\hn})^{-s} \! \left[ u(t,\cdot) \right] \dg \, ,
	$$
    so that
	$$
	\sum_{k =1}^{n-1}  h \, \int_{\hn} \partial_t \psi ({t}_k,\cdot) \, (-\Delta_{\hn})^{-s}  \tilde{u}_k \, \dg
	\underset{n \to \infty}{\longrightarrow}  \int_0^T \int_{\hn} \partial_t \psi \, (-\Delta_{\hn})^{-s} u \, \dg \, {\rm d}t \, .
	$$
     Hence, by combining the above results, we can deduce that
	\begin{equation}\label{Riemann-1}
	I_1 \underset{n \to \infty}{\longrightarrow}  \int_0^T \int_{\hn} \partial_t \psi \, (-\Delta_{\hn})^{-s} u \, \dg \, {\rm d}t \, .
	\end{equation}
	A similar (even simpler) computation, exploiting the fact that also $ \{ u^m_n \} $ converges to $ u^m $ uniformly in $  L^1(M)$ thanks to \eqref{eq-wm-1}, yields
	\begin{equation}\label{Riemann-2}
	I_2 \underset{n \to \infty}{\longrightarrow}  \int_0^T \int_{\hn}  u^m \, \psi \, \dg \, {\rm d}t \, .
	\end{equation}
	Since $\psi$ is compactly supported in time, it is apparent that $ \psi_0 =0  $ and $ \psi_{n-1}=0 $ for sufficiently large $n$, whence
	\begin{equation}\label{Riemann-3}
	\int_{\hn} \, [\psi_{n-1} \, (-\Delta_{\hn})^{-s} u_n  -  \psi_0 \, (-\Delta_{\hn})^{-s} u_0] \, \dg  \underset{n \to \infty}{\longrightarrow}   0 \, .
	\end{equation}
	As a consequence of \eqref{Riemann-1}, \eqref{Riemann-2} and \eqref{Riemann-3}, we finally obtain
	$$
	 \int_0^T \int_{\hn} \partial_t \psi \, (-\Delta_{\hn})^{-s} u \, \dg \, {\rm d}t
	- \int_0^T \int_{\hn}  u^m \, \psi \, \dg \, {\rm d}t = 0\, .
	$$
Given the arbitrariness of $ T $ and $\psi$, the above identity shows that $u$ is indeed a WDS, recalling moreover that it is bounded, thus $ u^m \in L^1((0, T) ; L^{1}_{loc}(\hn) ) $, and $L^1(\hn)$ is included (with continuity) in  $ L^1_{x_0,\GM}(\hn) $ for every $x_0 \in \hn$, so that $u$ also belongs to $ C^0([0,+\infty);L^1_{x_0,\GM}(\hn)) $.
	
\end{proof}

 \subsection{Basic properties of WDS for $ {L^1(M) \cap L^\infty(M)} $ initial data}

In this subsection, we provide two crucial propositions yielding stability estimates and inequalities, for approximate WDS, that will be fundamental in order to prove our main results in Section \ref{sect.semigroups}.

In the sequel, for the sake of simplicity, when referring to ``the'' WDS  to \eqref{NFDE} corresponding to $ L^1(M) \cap L^\infty(M) $ initial data, we will implicitly mean the one constructed in Propositions \ref{BCPP-Thm} and \ref{thm-weak-mild}, which therefore enjoys all the properties stated therein.

\begin{prop}\label{crucial}
Let $ \hn $ satisfy Assumption \ref{general}. Let $u$ be the WDS to \eqref{NFDE} corresponding to any nonnnegative initial datum $u_0 \in L^1(\hn) \cap L^\infty(M) $. Then, we have
\begin{equation}\label{estimate-1}
\int_{\hn} u(t, x)  \, \GM (x, x_0) \, \dg(x) \leq \int_{\hn} u_0(x) \,  \GM(x, x_0) \, \dg(x) \qquad \text{for all $t \ge 0 $ and all $ x_0 \in\hn $} \, ,
\end{equation}
and
\begin{align}\label{main-estimate-1}
 \left( \frac{t_0}{t_1} \right)^{\frac{m}{m-1}} (t_1 - t_0)\, u^m (t_0, x_0)
 &\leq
  \int_{\hn} \left[ u(t_0, x)-u(t_1, x) \right] \GM(x, x_0) \, \dg(x)\leq (m-1) \, \frac{t^{\frac{m}{m-1}}}{t_0^{\frac{1}{m-1}}} \, u^m(t, x_0)
\end{align}
for a.e.~$ (t_0,t_1,t,x_0) \in (\mathbb{R}^+)^3 \times \hn $, with $0<t_0\leq t_1\leq t$. 
\end{prop}
\begin{proof}
We follow closely the lines of proof of \cite[Proposition~4.2]{BV2} and \cite[Proposition 3.3]{BBGG}, where bounded Euclidean sets and the hyperbolic space are treated, respectively.
	
In order to prove \eqref{estimate-1}, we first notice that the following identity holds:
\begin{align}\label{limit-1}
&\int_{\hn} u(t_0, x) \, (- \Delta_{\hn})^{-s} \psi(x) \, \dg(x)
- \int_{\hn} u(t_1, x) \, (- \Delta_{\hn})^{-s} \psi(x) \, \dg(x)  \notag \\
=&  \int_{t_0}^{t_1} \int_{\hn}  u^{m}(t, x) \, \psi(x) \, \dg(x) \, {\rm d}t \ge 0 \,,
\end{align}
for every nonnegative $\psi \in L^{\infty}_{c}(\hn)$, and all $ 0 <t_0<t_1 $. The proof of \eqref{limit-1} follows by using \eqref{def_eq} with the   sequence of test functions (with separate variables)
 $\psi(t, x) = \phi_{n}(t) \psi(x),$ where $ \{\phi_{n}\} \subset C^1_c((0, +\infty))$ is  {in turn}   a suitable sequence of time-dependent functions such that $\phi_{n}(t)  \rightarrow \chi_{[t_0, t_1]} (t) $ a.e.~in $ (0,+\infty) $ and  $ \frac{\rm d }{{\rm d} t}  \phi_n \rightarrow \delta_{t_0} - \delta_{t_1} $, as $ n \to \infty $. Then, equation \eqref{limit-1} just follows by passing to the limit as $n\rightarrow \infty$, recalling the continuity properties of $ t \mapsto (- \Delta_{\hn})^{-s} u(t,\cdot) $ (see Remark \ref{remdef}) and exploiting Fubini's theorem.

Now fix $x_0 \in \hn$, and consider the sequence
$$
\psi_{n}^{(x_0)} :=\frac{\chi_{B_{1/n}(x_0)}}{\mu_M\!\left( B_{1/n}(x_0) \right)} \qquad \forall n \in \mathbb{N} \, .
$$
Clearly, each  $\psi^{(x_0)}_{n} $ is nonnegative and belongs to $ L^{\infty}_{c}(\hn)$, thus it is an admissible test function in \eqref{limit-1}. Furthermore, since for every $ x \in \hn $ the map  $ y \mapsto \GM(y,x)$ is continuous in $ \hn \setminus \{ x \} $ (recall the discussion in Subsection \ref{GAC}) and $\psi_{n}^{(x_0)} \rightarrow \delta_{x_0}$ as $n \to \infty$ in the sense of Radon measures, we obtain
\begin{equation}\label{boundgreenn-0}
(- \Delta_{\hn})^{-s}\big(\psi_{n}^{(x_0)} \big)(x) = \frac{1}{ \mu_M \! \left( B_{1/n}(x_0) \right)} \int_{B_{1/n}(x_0)} \GM(y, x)\,\dg(y) \underset{n \to \infty}{\longrightarrow}  \GM(x_0,x) = \GM(x,x_0) \, ,
\end{equation}
for every $ x \neq x_0 $. 
On the other hand, from Lemma \ref{lem.Phi.estimates} with $\sigma={1}/{n}$ and the bound $\big\| \psi_{n}^{(x_0)} \big\|_{\infty} \leq Cn^{N}$ (which follows from \eqref{non-coll} or simply the local Euclidean structure of any smooth Riemannian manifold), we deduce that there exists  $\overline C=\overline C(N,k,c,s)>0$ such that
\begin{equation}\label{boundgreenn}
(- \Delta_{\hn})^{-s}\big(\psi_{n}^{(x_0)} \big)(x) \leq \overline C \, \GM(x,x_0)  \qquad   \forall x \neq x_0 \, , \ \forall n \in \mathbb{N} \, .
\end{equation}
By virtue of \eqref{green-euc}, we know that $\GM(\cdot, x_0) \in L^q(B_R(x_0))$ for all $R>0$ and $1 \le q < \frac{N}{N -2s}$, so that by \eqref{boundgreenn-0}, \eqref{boundgreenn} and dominated convergence we infer that
\begin{equation*}
(- \Delta_{\hn})^{-s} \big( \psi_{n}^{(x_0)} \big) (x) \underset{n \to \infty}{\longrightarrow}  \GM( x , x_0 )
 \qquad \mbox{in} \ L^{q}(B_R(x_0))
\end{equation*}
for all such $q$. Hence, for every fixed $R \ge  1$ and $ \tau>0 $, it holds
\begin{align*}
&\left| \int_{B_R(x_0)} u(\tau, x) \, (- \Delta_{\hn})^{-s} \big( \psi_{n}^{(x_0)} \big)(x) \, \dg(x)
-  \int_{B_R(x_0)} u(\tau, x) \, \GM(x, x_0) \, \dg(x)\right| \\
 \leq & \left\| u(\tau) \right\|_{L^{q'}(\hn)}
 \left\| (- \Delta_{\hn})^{-s} \big( \psi_{n}^{(x_0)} \big) - \GM( \cdot , x_0) \right\|_{L^{q}(B_R(x_0))} \underset{n \to \infty}{\longrightarrow}  0 \, .
\end{align*}
Moreover, we also have that
$$
\left| \int_{\hn \setminus B_R(x_0)} u(\tau, x) \, (- \Delta_{\hn})^{-s} \big( \psi_{n}^{(x_0)} \big)(x) \, \dg(x)
-  \int_{\hn\setminus B_R(x_0)} u(\tau, x) \, \GM(x, x_0) \, \dg(x) \right| \underset{n \to \infty}{\longrightarrow}  0 \, ;
$$
indeed, by \eqref{boundgreenn} and the fact that $u (\tau, \cdot) \in L^1_{x_0,{\GM}}(\hn)$, it holds
$$
u(\tau, x) \, (- \Delta_{\hn})^{-s}\big(\psi_{n}^{(x_0)} \big)(x) \leq \overline C \, u(\tau, x) \, \GM(x,x_0) \in L^1(\hn \setminus B_R(x_0)) \quad  \text{for a.e.~} x \neq x_0 \, , \ \forall n \in \mathbb{N} \, ,
$$
whence the claim, still by \eqref{boundgreenn-0} and dominated convergence. As a result, we can apply \eqref{limit-1} with $\psi=\psi_{n}^{(x_0)}$ and let $ n \to \infty $, to get
$$
\int_{\hn} u(t_1, x)  \, \GM (x, x_0) \, \dg(x) \leq \int_{\hn} u(t_0,x) \,  \GM(x, x_0) \, \dg(x)  \, ,
$$
namely inequality \eqref{estimate-1} upon choosing $ t_1=t $ and letting $ t_0 \to 0 $ (it is readily seen that the r.h.s.~is actually stable).

Finally, we omit the proof of \eqref{main-estimate-1}, as it follows by repeating the arguments of \cite[Proposition 3.3 -- Step 3 and Step 4]{BBGG}, with inessential changes (just note that the starting point is still \eqref{limit-1} plus the monotonicity property \eqref{mon-est.OLD}). 
\end{proof}

\begin{prop}\label{monotonePhi}
	Let $ \hn $ satisfy Assumption \ref{general}. Let $u$ be the WDS to \eqref{NFDE} corresponding to any nonnnegative initial datum $u_0 \in L^1(\hn) \cap L^\infty(M) $. Then, we have
\begin{align}\label{newmonotonicity2}
\left\| u(t)\right\|_{L^1_{x_0,\GM}} \leq C \left\| u_0\right\|_{L^1_{x_0,\GM}}  \qquad \text{for all $t \ge 0$ and all $ x_0  \in \hn $} \, ,
\end{align}
for some $C=C(N,k,c,s)>0$. More in general, if $ u,v $ are two \emph{ordered} WDS to problem \eqref{NFDE} corresponding to nonnegative initial data $u_0 , v_0 \in L^1(\hn)\cap L^\infty(\hn)$, respectively, it holds
\begin{align}\label{newmonotonicity2-bis}
\left\| u(t) - v(t) \right\|_{L^1_{x_0,\GM}} \leq C \left\| u_0 - v_0 \right\|_{L^1_{x_0,\GM}}  \qquad \text{for all $t \ge 0$ and all $ x_0  \in \hn $} \, .
\end{align}
Furthermore, for all $0<R\leq 1$, we have
\begin{align}\label{newmonotonicity}
R^{N-2s}\int_{\hn\setminus B_R(x_0)} u(t,x) \, \GM(x,x_0) \, \dg(x) \leq C \left\| u(t)\right\|_{L^1_{x_0,\GM}} \qquad \text{for all $t \ge 0$ and all $ x_0  \in \hn $} \, ,
\end{align}
for another $C=C(N,k,c,s)>0$.
\end{prop}
\begin{proof}
From \eqref{limit-1} we know that
$$ \int_{\hn} u(t, x) \, (- \Delta_{\hn})^{-s} \psi (x) \, \dg(x) \leq \int_{\hn} u(t_0, x) \, (- \Delta_{\hn})^{-s}  \psi (x) \, \dg(x)$$
for all $t \geq t_0>0$ and all nonnegative $\psi \in L^{\infty}_{c}(\hn)$. Therefore, by Lemma \ref{potentials} we readily deduce that
$$
\left\| u(t) \right\|_{L^1_{x_0,\GM}} \leq C \left\| u(t_0) \right\|_{L^1_{x_0,\GM}}
$$
for some $C=C(N,k,c,s)>0$. Estimate \eqref{newmonotonicity2} just follows by letting $t_0 \rightarrow 0$, recalling that $ u(t) $ is continuous as a curve with values in $ L^1_{x_0,\GM}(M) $.

As far as \eqref{newmonotonicity2-bis} is concerned, it is enough to observe that by taking the difference between \eqref{limit-1} and the same identity applied to $v$ we have
\begin{align*}
&\int_{\hn} \left[u(t_0, x)-v(t_0,x) \right] (- \Delta_{\hn})^{-s} \psi(x) \, \dg(x)
- \int_{\hn} \left[ u(t_1, x)-v(t_1,x) \right] (- \Delta_{\hn})^{-s} \psi(x) \, \dg(x)  \notag \\
=&  \int_{t_0}^{t_1} \int_{\hn}  \left[ u^{m}(t, x) -v^{m}(t, x)  \right] \psi(x) \, \dg(x) \, {\rm d}t \ge 0 \,,
\end{align*}
where we assume that $ u \ge v $ (otherwise one swaps the two solutions). Then, we repeat exactly the same argument as above.

In order to prove \eqref{newmonotonicity}, we pick $\psi(x)=\psi^{(x_0)}(x)=\chi_{B_{1/2}(x_0)}(x) $. By Lemma \ref{lem.Phi.estimates} (with $\sigma=1/2$), we infer that there exist three positive constants $\underline C, C_1, C_2$, only dependent on $N,k,c,s$, such that for every $0<R\leq 1$ it holds
\begin{align*}
& \, R^{N-2s} \, \underline C \int_{\hn\setminus B_R(x_0)} u(t, x) \,  \GM(x,x_0) \, \dg(x) \leq \int_{\hn} u(t, x) \, (- \Delta_{\hn})^{-s}\big(\psi^{(x_0)}\big) (x) \, \dg(x) \\
\leq & \, \int_{B_2(x_0)} u(t, x) \, (- \Delta_{\hn})^{-s}\big(\psi^{(x_0)}\big) (x) \, \dg(x) +  \int_{\hn\setminus B_2(x_0)} u(t, x) \, (- \Delta_{\hn})^{-s}\big(\psi^{(x_0)}\big) (x) \, \dg(x) \\
\leq & \, C_1 \int_{B_2(x_0)} u(t, x) \, \dg(x)+ C_2 \int_{\hn\setminus B_2(x_0)} u(t, x)\,  \GM(x,x_0) \, \dg(x) \, ,
\end{align*}
where for the upper bounds we have exploited Remark \ref{generalR}. Estimate \eqref{newmonotonicity}, finally, follows upon noticing that
 $$\int_{B_2(x_0)\setminus B_1(x_0)} u(t, x) \, \dg(x) \leq \frac{1}{\kappa} \int_{B_2(x_0)\setminus B_1(x_0)}\GM(x,x_0)\, u(t, x) \, \dg(x) $$
with $ \kappa := \displaystyle{\min_{1\leq r(x,x_0)\leq 2}} \GM(x,x_0)$ (recall the uniform lower bound \eqref{green-est-low}).
\end{proof}

\section{Proof of Theorems \ref{thm-existence}, \ref{thm.smoothing.HN-like} and \ref{smoothPhi}} \label{sect.semigroups}
Following a similar strategy to \cite{BV2},  all the main results will first be established for approximate WDS (see Section \ref{smooth}) and then extended to general WDS, by means of a (relatively) standard limiting process relying on the monotone approximation of a nonnegative initial datum $ u_0$, that merely belongs to  $ L^1_{\GM}(\hn) $, with a sequence of ``truncated'' data $ u_{0,n} \in L^1(M) \cap L^\infty(M) $.

Before, we provide an anticipation of a smoothing effect, that will be needed only in the proof of Theorem \ref{thm-existence} when passing to the limit in the second term on the l.h.s.~of \eqref{def_eq}, which requires (at least) boundedness $ L^m_{loc}(M) $. Moreover, certain estimates provided along the proof will also serve as key starting points for establishing the ``real'' smoothing effects stated in Theorems \ref{thm.smoothing.HN-like} and \ref{smoothPhi}.
\begin{prop}\label{smoothPhi-lemma}
Let $ \hn $ satisfy Assumption \ref{general}. Let $u$ be the WDS to \eqref{NFDE} corresponding to any nonnegative initial datum $u_0 \in L^1(\hn) \cap L^\infty(M) $. Then, there exists $ C = C (N,k,c,s,m)>0$ such that
	\begin{equation}\label{thm.smoothing.HN-like.estimate3-lemma}
	\left\| u(t) \right\|_{L^\infty(M)} \le  C \left( \frac{\left\| u_0\right\|_{L^1_{\GM}}^{2s \vartheta_1}}{t^{N \vartheta_1}} \vee \left\| u_0\right\|_{L^1_{\GM}} \right) \qquad \forall t > 0 \, .
	\end{equation}
\end{prop}
\begin{proof}  We first comment that, in the following, inessential numerical constants will always be denoted by $C$.
	The key ingredient is the following estimate, which is a simple consequence of \eqref{main-estimate-1} with the choice $t_1 = 2t_0$ (let $ t_0>0 $ be arbitrary but fixed):
	\begin{align}\label{thm.smoothing.HN-like.proof.1}
	u^m&(t_0, x_0) 
	\leq \frac{2^{\frac{m}{m-1}}}{ t_0} \int_{\hn} u(t_0, x) \, \GM(x, x_0) \, \dg(x)\\
	&= \underbrace{\frac{ 2^{\frac{m}{m-1}}}{ t_0}  \int_{B_{R}(x_0)}   u(t_0, x) \, \GM(x, x_0)  \, \dg(x)}_{(I)}
	+  \underbrace{\frac{  2^{\frac{m}{m-1}}}{ t_0} \int_{\hn\setminus B_{R}(x_0)}  \, u(t_0, x)  \, \GM(x, x_0) \, \dg(x)}_{(II)} \, , \nonumber
	\end{align}
	for all $R>0$. Note that \eqref{thm.smoothing.HN-like.proof.1}, actually, holds for almost every $ t_0>0 $ and $ x_0 \in \hn $, but this is not an issue to our purposes given the time continuity properties of $u$ and the fact that we take essential suprema (even though we will not mention it explicitly).
	
	In order to handle $ (I) $, first we prove the following Green function estimate, valid under our running assumptions:
	\begin{equation}\label{Hyp.Green.RN.01}
	\int_{B_R(x_0)}\GM(x,x_0) \, \dg(x) \le C R^{2s} \qquad \text{for all } 0<R\le 1 \, ,
	\end{equation}
	for some $C=C(N,k,c,s)>0$. To this aim, let $\mathcal{H}_M^{N-1}$ denote the $(N-1)$-dimensional Hausdorff measure on $ M$. As a consequence of the  celebrated Bishop-Gromov theorem (see e.g.~\cite[Theorem 1.1]{H}), it is readily seen that $\mathcal{H}_M^{N-1}(\partial B_r(x_0))\le C r^{N-1}$ for every $r\in(0,1]$ and every $x_0\in M$, where $C$ is a suitable positive constant depending only on $N$ and the curvature bound \eqref{ric}. Therefore, taking advantage of \eqref{green-euc} and the coarea formula we infer, for all $0<R\le1$:
	\begin{equation*}\label{greenint}\begin{aligned}
	\int_{B_R(x_0)}\GM(x,x_0) \, \dg(x) &\le C\int_{B_R(x_0)}\,\frac{1}{ r(x,x_0)^{N-2s}}\, \dg(x)\\
	& =  C\int_0^{R}\frac1{r^{N-2s}}\,\mathcal{H}_M^{N-1}(\partial B_r(x_0)) \, {\textrm d}r\le C R^{2 s} \, ,
	\end{aligned}
	\end{equation*}
	which is precisely the claim. Having established \eqref{Hyp.Green.RN.01}, it follows that
	\begin{align}\label{thm.smoothing.HN-like.proof.2}
	(I)&\le \frac{1}{m} \left\| u(t_0) \right\|_{L^\infty(\hn)}^m + \frac{C}{t_0^{\frac{m}{m-1}}} \left( \int_{B_{R}(x_0)} \GM(x, x_0) \, \dg(x) \right)^{\frac{m}{m-1}} \\
	&\le \frac{1}{m} \left\|u(t_0)\right\|_{L^\infty(\hn)}^m+\frac{C}{t_0^{\frac{m}{m-1}}}	\, R^{\frac{2sm}{m-1}} \qquad \text{for all } 0 < R \le 1 \,,\nonumber
	\end{align}
	where $C>0$ is another constant depending only on $N,k,c,s,m$.
	
	In order to bound $(II)$, note that from \eqref{newmonotonicity} we know that for all $x_0\in \hn$ and all $0<R\leq 1$ it holds
	$$
	R^{N-2s}\int_{\hn\setminus B_R(x_0)} u(t_0,x) \, \GM(x,x_0) \, \dg(x) \leq C \left\| u(t_0) \right\|_{L^1_{x_0,\GM}}  .
	$$
	Hence, going back to \eqref{thm.smoothing.HN-like.proof.1} and \eqref{thm.smoothing.HN-like.proof.2}, we deduce that
	\begin{align*}
	u^m(t_0, x_0)	\leq  \frac{1}{m} \left\| u(t_0) \right\|_{L^\infty(M)}^m + \frac{C}{t_0^{\frac{m}{m-1}}}  \, R^{\frac{2sm}{m-1}}  +  \frac{C}{t_0\,R^{N-2s}} \left\| u(t_0) \right\|_{L^1_{x_0,\GM}} ,
	\end{align*}
	still for all $ 0 < R \le 1 $. Taking the supremum over $x_0\in M$ and recalling the definition of the norm $ \| \cdot \|_{L^1_{\GM}} $, we thus obtain
	\begin{equation}\label{greenR}
	\begin{aligned}
	\left\| u(t_0) \right\|_{L^\infty(M)}^m \le C\, \frac{R^{\frac{2sm}{m-1}}}{t_0^{\frac{m}{m-1}}} \left( 1+ \frac{t_0^{\frac{1}{m-1}}  \left\| u(t_0) \right\|_{L^1_{\GM}}  }{R^{\frac{1}{(m-1)\vartheta_1}}}\right) .
	\end{aligned}
	\end{equation}
	Note that both sides are finite since $ u(t_0) \in L^1(M) \cap L^\infty(M) $, and $ L^1(M) $ is included with continuity in $ L^1_{\GM}(M) $. By choosing
	\begin{equation}\label{greenR-nn}
	 R=\left(t_0^{\frac{1}{m-1}} \left\| u(t_0) \right\|_{L^1_{\GM}}  \right)^{(m-1)\vartheta_1}
\end{equation}
	 and plugging it into \eqref{greenR}, we end up with the following bound:
	 \begin{equation}\label{greenR-bis}
	 \left\| u(t_0) \right\|_{L^\infty(M)} \le  C  \, \frac{\left\| u(t_0) \right\|_{L^1_{\GM}}^{2s \vartheta_1}}{t_0^{N \vartheta_1}} \le C  \, \frac{\left\| u_0 \right\|_{L^1_{\GM}}^{2s \vartheta_1}}{t_0^{N \vartheta_1}} \, ,
	 \end{equation}
	 where in the last passage we have exploited the stability estimate \eqref{newmonotonicity2} (up to taking the supremum over $x_0 \in \hn $). We point out that \eqref{greenR-bis} holds subject to $ R \le 1 $ in \eqref{greenR-nn}, which is achieved provided
	 \[
	 \left(t_0^{\frac{1}{m-1}} \left\| u_0\right\|_{L^1_{\GM}}  \right)^{(m-1)\vartheta_1}\le C \, ,
	 \]
  by virtue of \eqref{newmonotonicity2} (where we recall that a multiplicative constant is present on the right-hand side bound).
	 
  Given the arbitrariness of $t_0$, we finally obtain
	 $$
	 \left\| u(t) \right\|_{L^\infty(M)} \le C  \, \frac{\left\| u_0 \right\|_{L^1_{\GM}}^{2s \vartheta_1}}{t^{N \vartheta_1}} \qquad  \text{for all }  t\le C \left\| u_0 \right\|_{L^1_{\GM}}^{-(m-1)}.
	 $$
	 It is readily checked that, at the borderline time $  t_\ast  =C \left\| u_0 \right\|_{L^1_{\GM}}^{-(m-1)}  $,  the above bound entails  $ \| u(t_\ast) \|_{L^\infty(M)} \le C \, \| u_0 \|_{L^1_{\GM}} $, whence \eqref{thm.smoothing.HN-like.estimate3-lemma} follows from \eqref{greenR-bis} recalling that $ L^p(M) $ norms are nonincreasing along the nonlinear semigroup constructed in Subsection \ref{NS-abstract}. 
	 	\end{proof}

\subsection{Proof of Theorem \ref{thm-existence}}
We proceed by (monotone) approximation in terms of the mild solutions introduced in Subsection \ref{NS-abstract}  corresponding $ L^1(M) \cap L^\infty(M) $ initial data, which by virtue of Proposition \ref{thm-weak-mild} are also WDS. More precisely, our weak dual solution corresponding to a nonnegative $ u_0 \in L^1_{\GM}(M)$ will be obtained as a limit of a pointwise monotone sequence of such mild solutions.

To our purposes, for every $ n \in \mathbb{N} $ we set
\begin{equation*}\label{monotone-approx-1}
u_{0,n} := \chi_{B_n(o)} \left(  u_0 \wedge n  \right) ,
\end{equation*}
where $ o \in \hn $ is any fixed point. Since each $ u_{0,n} $ is bounded and compactly supported, it is apparent that $u_{0,n}\in L^1(\hn) \cap L^{\infty}(\hn)$, and by construction $u_{0,n} \leq u_{0,n+1} \leq u_{0} $. If we let $ u_n = u_n(t,x) $ denote the (mild and) WDS to \eqref{NFDE} with initial datum $ u_{0,n} $, in view of the comparison principle entailed by \eqref{comparison.L1} we deduce that $u_n(t,x)\le u_{n+1}(t,x)$ for every $ t>0 $ and a.e.~$ x \in \hn $, namely also the sequence of solutions $ \{ u_n \} $ is monotone increasing. As a result, the pointwise limit
\begin{equation}\label{def.WDS.approx}
u(t,x):=\lim_{n\to\infty}u_{n}(t,x)
\end{equation}
exists, and it is therefore our candidate to be the weak dual solution starting from $ u_0 $.  What follows aims at showing that $u$ is indeed the sought solution.

\noindent\textbf{Convergence occurs in $ C^0([0,T]; L^1_{x_0,\GM}(M)) $.} Thanks to the stability estimate \eqref{newmonotonicity2}, we have that
$$
\left\| u_n(t)\right\|_{L^1_{x_0,\GM}} \leq C \left\| u_{0,n}\right\|_{L^1_{x_0,\GM}}  \qquad \text{for all $t \ge 0$ and all $ x_0  \in \hn $} \, .
$$
By the monotone convergence theorem, it is plain that $ u_{0,n} \to u_0 $ in $ L^1_{x_0,\GM}(M) $ as $ n \to \infty $, so that by Fatou's lemma we obtain
\begin{equation}\label{obv}
\left\| u(t)\right\|_{L^1_{x_0,\GM}} \leq C \left\| u_{0}\right\|_{L^1_{x_0,\GM}}  \qquad \text{for all $t \ge 0$ and all $ x_0  \in \hn $} \, ,
\end{equation}
which in particular yields $ u(t,\cdot) \in L^1_{x_0,\GM}(M) $ for all $t>0$. Hence, again by monotone convergence, we deduce that $ u_{n}(t,\cdot) \to u(t,\cdot) $ in $ L^1_{x_0,\GM}(M) $ as $ n \to \infty $. Therefore, if we show that the sequence $ \{ u_n \} $ is uniformly equicontinuous in $ C^0([0,T]; L^1_{x_0,\GM}(M)) $, for every fixed $ T>0 $, by Ascoli-Arzelà theorem we can infer that convergence actually takes place in such space, and in particular $ u \in C^0([0,T]; L^1_{x_0,\GM}(M))  $ as required in Definition \ref{defi_WDS}.

To this end, let $ \varepsilon>0 $ be fixed. We pick $ n_0 \in \mathbb{N} $ so large that
$$
\left\| u_{0} - u_{0,n_0} \right\|_{L^1_{x_0,\GM}}  \le \frac{\varepsilon}{4C} \, ,
$$
where $C>0$ is the constant appearing in \eqref{newmonotonicity2-bis}. Since $ t \mapsto u_{n_0}(t,\cdot) $ is a (uniformly) continuous curve with values in $ L^1_{x_0,\GM}(M) $, there exists $ \delta>0 $ such that $ \left\| u_{n_0}(t) - u_{n_0}(s) \right\|_{L^1_{x_0,\GM}}  \le \frac \varepsilon 2 $ for every $ t,s \ge 0 $ with $ |t-s| \le \delta $. Hence, by virtue of the triangle inequality and the stability estimate \eqref{newmonotonicity2-bis} (recall that the sequence $ \{ u_n \} $ is ordered), we obtain:
$$
\begin{aligned}
\left\| u_n(t) - u_n(s) \right\|_{L^1_{x_0,\GM}} \! \le & \left\| u_{n}(t) - u_{n_0}(t) \right\|_{L^1_{x_0,\GM}}  \! + \left\| u_{n_0}(t) - u_{n_0}(s) \right\|_{L^1_{x_0,\GM}} \! + \left\| u_n(s) - u_{n_0}(s) \right\|_{L^1_{x_0,\GM}} \\
\le & \, 2C \left\| u_{0,n} - u_{0,n_0} \right\|_{L^1_{x_0,\GM}} \! + \left\| u_{n_0}(t) - u_{n_0}(s) \right\|_{L^1_{x_0,\GM}}  \! \le \frac{\varepsilon}{2} + \frac{\varepsilon}{2} = \varepsilon \, ,
\end{aligned}
$$
for every $ n \ge n_0 $ and $ t,s \ge 0 $ with $ |t-s| \le \delta $. This proves uniform equicontinuity in $ L^1_{x_0,\GM}(M) $, so that $ t \mapsto u(t,\cdot) $ is also a continuous curve with values in $ L^1_{x_0,\GM}(M) $ satisfying $ u(0,\cdot)=u_0 $.

\noindent\textbf{Convergence occurs in $ L^m((0,T); L^m_{loc}(M)) $.} Thanks to the smoothing effect established in Proposition \ref{smoothPhi-lemma}, along with the fact that $ u_{0,n} \le u_0 $, we have:
	\begin{equation}\label{thm.smoothing.HN-like.estimate3-lemma-proof}
\left\| u_n(t) \right\|_{L^\infty(M)} \le  C \left( \frac{\left\| u_{0,n}\right\|_{L^1_{\GM}}^{2s \vartheta_1}}{t^{N \vartheta_1}} \vee \left\| u_{0,n}\right\|_{L^1_{\GM}} \right) \le C \left( \frac{\left\| u_0\right\|_{L^1_{\GM}}^{2s \vartheta_1}}{t^{N \vartheta_1}} \vee \left\| u_0\right\|_{L^1_{\GM}} \right) \qquad \forall t > 0 \, ,
\end{equation}
for every $ n \in \mathbb{N} $. Because the $ L^\infty(M) $ norm is lower semicontinuous w.r.t.~pointwise convergence, upon passing to the limit in \eqref{thm.smoothing.HN-like.estimate3-lemma-proof} as $n \to \infty$, we end up with
	\begin{equation*}\label{thm.smoothing.HN-like.estimate3-lemma-proof-bis}
\left\| u(t) \right\|_{L^\infty(M)} \le \lim_{n \to \infty} \left\| u_n(t) \right\|_{L^\infty(M)}  \le C \left( \frac{\left\| u_0\right\|_{L^1_{\GM}}^{2s \vartheta_1}}{t^{N \vartheta_1}} \vee \left\| u_0\right\|_{L^1_{\GM}} \right) \qquad \forall t > 0 \, .
\end{equation*}
As a result, for every $ r>0 $ and $T>0$ we have:
$$
\begin{aligned}
\int_{0}^{T} \int_{B_r(o)}  u^{m}(t, x) \, \dg(x) \, {\rm d}t  \le & \, \int_{0}^{T} \left\| u(t) \right\|_{L^\infty(M)}^{m-1} \int_{B_r(o)}  u(t, x) \, \dg(x) \, {\rm d}t \\
\le & \, C \int_{0}^{T} \left( \frac{\left\| u_0\right\|_{L^1_{\GM}}^{2s (m-1) \vartheta_1}}{t^{N (m-1) \vartheta_1}} \vee \left\| u_0\right\|_{L^1_{\GM}}^{m-1} \right) \int_{B_r(o)}  u(t, x) \, \dg(x) \, {\rm d}t \, .
\end{aligned}
$$
Because $ N(m-1)\vartheta_1 < 1 $ and  $ L^1_{\GM}(M) $ is included with continuity in $ L^1(B_r(o)) $, from the above estimate and \eqref{obv} (up to taking the $\sup$ over $ x_0 \in \hn $), we deduce that $ u^m \in L^1((0,T); L^1_{loc}(M)) $, or equivalently $ u \in L^m((0,T); L^m_{loc}(M))  $.

Again, by monotonicity, this shows in particular that $  u_n \to u $ also in $L^m((0,T); L^m_{loc}(M))$.

\noindent\textbf{Passing to the limit in the weak dual formulation.} Having just shown that $ u $ belongs to the required functional spaces and $ u(0,\cdot)=u_0 $, we are just left with proving that it satisfies the weak dual formulation \eqref{def_eq}.

Since each approximate solution $ u_n $ is a WDS, we know that
\begin{equation}\label{def_eq_approx}
\int_{0}^{T} \int_{\hn}  \partial_t \psi \,  (- \Delta_{\hn})^{-s} u_n \, \dg \, {\rm d}t - \int_{0}^{T} \int_{\hn} u^m_n \, \psi \, \dg\, {\rm d}t = 0 \qquad \forall n \in \mathbb{N} \, ,
\end{equation}
for every $ T>0  $ and every $\psi \in C^1_c((0,T); L_c^{\infty}(\hn))$. Thanks to the convergence of $  u_n \to u $ in $L^m((0,T); L^m_{loc}(M))$, we can safely pass to the limit in the second term on the l.h.s.~of \eqref{def_eq_approx} to get
$$
\int_{0}^{T} \int_{\hn} u^m_n \, \psi \, \dg\, {\rm d}t  \underset{n \to \infty}{\longrightarrow} \int_{0}^{T} \int_{\hn} u^m \, \psi \, \dg\, {\rm d}t \, .
$$
As for the first term, we already know that $ u \in C^0([0,T]; L^1_{x_0,\GM}(M)) $; therefore, in the light of Remark \ref{remdef}, we have that $ (- \Delta_{\hn})^{-s} u \in  C^0([0,T]; L^1_{loc}(M))  $. By the definition of the operator $  (- \Delta_{\hn})^{-s} $, it is apparent that also $ \{ (- \Delta_{\hn})^{-s} u_n \} $ is a monotone increasing sequence pointwise converging to  $(- \Delta_{\hn})^{-s} u$, thus convergence surely occurs in $ L^1((0,T); L^1_{loc}(M))   $, whence
$$
\int_{0}^{T} \int_{\hn}  \partial_t \psi \,  (- \Delta_{\hn})^{-s} u_n \, \dg \, {\rm d}t  \underset{n \to \infty}{\longrightarrow} \int_{0}^{T} \int_{\hn}  \partial_t \psi \,  (- \Delta_{\hn})^{-s} u \, \dg \, {\rm d}t \, ,
$$
which this completes the proof. \qed

\begin{proof}[Proof of Corollary \ref{unique}]
The argument is an adaptation to the present setting of the proof of Theorem 4.5 in \cite{BV1}, so we sketch it here for the reader's convenience.

Let $ \{ u_n \} $ be the monotone increasing sequence of approximate solutions introduced in the proof of Theorem \ref{thm-existence}, and consider any other monotone increasing sequence $ \{ v_{0,k} \}  \subset L^1(\hn) \cap L^\infty(\hn)$, with $ v_{0,k} \ge 0 $, such that $  v_{0,k} \to u_0 $ pointwise as $ k \to \infty $. Upon repeating exactly the same construction as in the just mentioned proof, it is readily seen that one obtains another weak dual solution $v$ to problem \eqref{NFDE}. We want to show that $u=v$. To this aim, estimate \eqref{newmonotonicity2-bis} applied with $ (u,v) = (u_n,v_k) $ yields
\begin{align}\label{newmonotonicity2-ter}
\left\| u_n(t) - v_k(t) \right\|_{L^1_{x_0,\GM}} \leq C \left\| u_{0,n} - v_{0,k} \right\|_{L^1_{x_0,\GM}}  \qquad \text{for all $t \ge 0$ and all $ x_0  \in \hn $} \, ,
\end{align}
 for every $ n,k \in \mathbb{N} $. Passing to the limit in \eqref{newmonotonicity2-ter}  first as $ n \to \infty $ and then as $ k \to \infty $, we thus end up with
 $$
 \left\| u(t) - v(t) \right\|_{L^1_{x_0,\GM}} = \lim_{k\to\infty} \, \lim_{n\to\infty} \left\| u_n(t) - v_k(t) \right\|_{L^1_{x_0,\GM}}  \le 0  \qquad \text{for all $t \ge 0$ and all $ x_0  \in \hn $} \, ,
 $$
namely the assertion.
\end{proof}

\subsection{Proof of Theorem \ref{thm.smoothing.HN-like}}

Without loss of generality, we can restrict ourselves to initial data in $ L^1(M) \cap L^\infty(M) $ and the corresponding mild solutions constructed in Section \ref{smooth}, since each of the smoothing effects \eqref{S-1}, \eqref{S-2} and \eqref{thm.smoothing.HN-like.estimate.2} is stable under the (monotone) passage to the limit \eqref{def.WDS.approx} (note that in this case $ u_n \to u $ in $ C^0([0,T];L^1(M)) $). 

In order to prove \eqref{S-1}, we go back to \eqref{thm.smoothing.HN-like.proof.1} and simply observe that, for all $ R>0 $, it holds
\begin{align}\label{thm.smoothing.HN-like.proof.3}
(II) \le \frac{C}{t_0} \int_{\hn\setminus B_{R}(x_0)} u(t_0,x) \, \frac{1}{r(x, x_0)^{N-2s}} \, \dg(x) \leq \frac{C}{t_0} \, \frac{\left\| u(t_0) \right\|_{L^1(\hn)}}{R^{N-2s}} \, ,
\end{align}
where we have used \eqref{green-euc}. By combining \eqref{thm.smoothing.HN-like.proof.3} and \eqref{thm.smoothing.HN-like.proof.2}, and taking the supremum over $x_0\in \hn$, we obtain
\begin{align}\label{thm.smoothing.HN-like.proof.5}
\left\| u(t_0) \right\|_{L^\infty(\hn)}^m \le C \, \frac{R^{\frac{2sm}{m-1}}}{t_0^{\frac{m}{m-1}}} \left( 1 + \frac{t_0^{\frac{1}{m-1}} \left\| u(t_0) \right\|_{L^1(\hn)}}{R^{\frac{1}{(m-1)\vartheta_1}}}\right)  \qquad \text{for all } 0<R\le 1 \, .
\end{align}
If we now choose
\begin{equation}\label{ee5}
R=\left(t_0^{\frac{1}{m-1}} \left\| u(t_0) \right\|_{L^1(M)}  \right)^{(m-1)\vartheta_1}
\end{equation}
and plug it into \eqref{thm.smoothing.HN-like.proof.5}, we infer that
\begin{equation}\label{ee6}
\left\| u(t_0) \right\|_{L^\infty(M)} \le  C  \, \frac{\left\| u(t_0) \right\|_{L^1(M)}^{2s \vartheta_1}}{t_0^{N \vartheta_1}} \le C  \, \frac{\left\| u_0 \right\|_{L^1(M)}^{2s \vartheta_1}}{t_0^{N \vartheta_1}} \, ,
\end{equation}
where in the last passage we have exploited the fact that the $ L^1(M) $ norm is nonincreasing in time (namely \eqref{eq-wm-1}). Note that \eqref{ee6} holds subject to $ R \le 1 $ in \eqref{ee5}, which is achieved provided
\[
\left(t_0^{\frac{1}{m-1}} \left\| u_0\right\|_{L^1(M)}  \right)^{(m-1)\vartheta_1}\le 1 \, ,
\]
using again the time monotonicity of the $ L^1(M) $ norm. Because, at the borderline time $  t_\ast  = \left\| u_0 \right\|_{L^1(M)}^{-(m-1)}  $, estimate \eqref{ee6} yields $\| u(t_\ast) \|_{L^\infty(M)} \le C \| u_0 \|_{L^1(M)}  $, the smoothing effect \eqref{S-1} follows upon noticing again that $ L^p(M) $ norms (in particular for $p=\infty$) are nonincreasing along the semigroup of mild solutions.

Under Assumption \ref{ch}, i.e.~when $M$ is a Cartan-Hadamard manifold, it suffices to observe that \eqref{thm.smoothing.HN-like.proof.5} is valid for \emph{all} $R>0$ as, in this case, estimate \eqref{Hyp.Green.RN} replaces \eqref{Hyp.Green.RN.01}. As a result, the previous strategy holds without any smallness constraint on $R$, and therefore on $t_0$ as well.

We finally turn to the proof of \eqref{thm.smoothing.HN-like.estimate.2}, under Assumption \ref{neg}. First of all, we recall estimate \eqref{Hyp.Green.HN2}, which in particular entails
\begin{equation}\label{Hyp.Green.HN1.01}
\int_{B_R(x_0) }\GM(x,x_0) \, \dg(x) \le C R^{s}  \qquad \forall R \ge 1 \, .
\end{equation}
Let us go back again to \eqref{thm.smoothing.HN-like.proof.1}, and properly bound the two items on the rightmost side. As for the first one, upon reasoning exactly as in \eqref{thm.smoothing.HN-like.proof.2} (using \eqref{Hyp.Green.HN1.01} instead), we get
\begin{align}\label{thm.smoothing.HN-like.proof.6}
(I) \le \frac{1}{m} \left\| u(t_0) \right\|_{L^\infty(\hn)}^m + \frac{C}{t_0^{\frac{m}{m-1}}} \, R^{\frac{sm}{m-1}} \qquad \forall R \ge 1 \, .
\end{align}
In order to handle the second one, it is convenient to exploit the pointwise bound \eqref{Hyp.Green.HN0b}:
\begin{equation}\label{thm.smoothing.HN-like.proof.7}
(II) \le \frac{C}{t_0} \int_{\hn \setminus B_{R}(x_0)} u(t_0, x) \, \frac{e^{-(N-1) \sqrt{ \c} \, r(x, x_0)} }{r(x,x_0)^{1-s}}  \, \dg(x)  \le \frac{C}{t_0} \, \frac{\left\| u(t_0) \right\|_{L^1(\hn)}}{R^{1-s} \, e^{(N-1) \sqrt{ \c} \, R}} \qquad \forall R \ge 1 \, .
\end{equation}
By combining \eqref{thm.smoothing.HN-like.proof.1}, \eqref{thm.smoothing.HN-like.proof.6} and \eqref{thm.smoothing.HN-like.proof.7}, recalling again that $\|u(t_0)\|_{L^1(\hn)}\le \|u_0\|_{L^1(\hn)}$ and taking the supremum over $ x_0 \in \hn $, we end up with
\begin{align} \label{thm.smoothing.HN-like.proof.9}
\left\| u(t_0) \right\|_{L^\infty(\hn)}^m \le C\,\frac{R^{\frac{sm}{m-1}}}{t_0^{\frac{m}{m-1}}} \left( 1+ \frac{t_0^{\frac{1}{m-1}}\|u_0\|_{L^1(\hn)}}{e^{(N-1) \sqrt{ \c} \,R}}\right) \qquad \forall R \ge 1 \,  ,
\end{align}
where in the second term inside parentheses we have discarded a power of $R$ at the denominator, since it does not contribute to any improvement of the final estimate. Choosing now
$$R=\frac{1}{(N-1) \sqrt{ \c}} \, \log\!\left(t_0^{\frac{1}{m-1}} \left\|u_0\right\|_{L^1(\hn)}\right)\geq 1 \qquad \Longleftrightarrow \qquad t_0 \ge e^{(N-1) (m-1)\sqrt{\c}} \left\|u_0\right\|_{L^1(\hn)}^{-(m-1)} , $$
inequality \eqref{thm.smoothing.HN-like.proof.9} gives \eqref{thm.smoothing.HN-like.estimate.2}, and this concludes the proof of Theorem \ref{thm.smoothing.HN-like}. \qed

\subsection{Proof of Theorem \ref{smoothPhi}}
Similarly to the proof of Theorem \ref{thm.smoothing.HN-like}, the idea is to first establish the smoothing estimates \eqref{thm.smoothing.HN-like.estimate3}--\eqref{thm.smoothing.HN-like.estimate4} on approximate mild solutions, and then pass to the limit by monotonicity. However, here we need to be more cautious, since as observed in Remark \ref{no-monotone-app} monotone increasing sequences in general do not converge strongly in the space $ L^1_{\GM}(M) $. For such reason, we prefer in this case to keep the explicit symbol $ u_n $ of the approximate mild solutions constructed in the proof of  Theorem \ref{thm-existence}, and then pass to the limit as $ n \to \infty $.

By the proof of Proposition \ref{smoothPhi-lemma}, we can infer that there exists $ C_1 = C_1(N,k,c,s,m)>0$ such that
\begin{equation}\label{LLL}
\left\| u_n(t) \right\|_{L^\infty(M)} \le C_1 \left( \frac{\left\| u_{n}(t)\right\|_{L^1_{\GM}}^{2s \vartheta_1}}{t^{N \vartheta_1}} \vee \left\| u_{0,n}\right\|_{L^1_{\GM}} \right) \qquad \forall t > 0 \, ,
\end{equation}
for every $ n \in \mathbb{N} $. Note that \eqref{LLL} is a direct consequence of \eqref{greenR-bis} and the same reasoning as in the end of the proof of Proposition \ref{smoothPhi-lemma}. Since $ u_{0,n} \le u_0 $ and $ u_n \le u $, we deduce that
\begin{equation*}\label{LLL-bis}
\left\| u_n(t) \right\|_{L^\infty(M)} \le C_1 \left( \frac{\left\| u(t)\right\|_{L^1_{\GM}}^{2s \vartheta_1}}{t^{N \vartheta_1}} \vee \left\| u_{0}\right\|_{L^1_{\GM}} \right) \qquad \forall t > 0 \, ,
\end{equation*}
whence
$$
 \left\| u(t) \right\|_{L^\infty(M)} \le \liminf_{n\to\infty} \left\| u_n(t) \right\|_{L^\infty(M)} \le C_1 \left( \frac{\left\| u(t)\right\|_{L^1_{\GM}}^{2s \vartheta_1}}{t^{N \vartheta_1}} \vee \left\| u_{0}\right\|_{L^1_{\GM}} \right) \qquad \forall t > 0 \, ,
$$
namely the leftmost inequality in \eqref{thm.smoothing.HN-like.estimate3}. As for the rightmost inequality, it is enough to notice that $ \| u(t) \|_{L^1_{\GM}} \le C \| u_0 \|_{L^1_{\GM}} $, which follows upon taking the supremum over $ x_0 \in \hn $ in \eqref{obv}.

We now deal with the proof of the long-time behavior, i.e.~\eqref{thm.smoothing.HN-like.estimate4}. To this aim, first of all we notice that \eqref{thm.smoothing.HN-like.proof.2} (with $ u=u_n $) holds for all $R>0$ under Assumption \ref{ch}, as a consequence of \eqref{Hyp.Green.RN}. Besides, for all $ R \ge 1 $ we have
\begin{equation*}
(II)\le \frac C{t_0} \left\| u_n(t_0)\right\|_{L^1_{x_0,\GM}} ,
\end{equation*}
this being a direct consequence of the definition of $ \| \cdot \|_{L^1_{x_0,\GM}} $. Therefore, by proceeding similarly to the proof of Proposition \ref{smoothPhi-lemma}, we end up with the following bound, valid for every $ R \ge 1 $ and (almost every) $ x_0 \in \hn $:
\begin{align*}
u_n^m(t_0, x_0)
&\leq  \frac{1}{m} \left\| u_n(t_0) \right\|_{L^\infty(M)}^m + \frac{C}{t_0^{\frac{m}{m-1}}}  \, R^{\frac{2sm}{m-1}}  + \dfrac{C}{t_0} \left\| u_n(t_0) \right\|_{L^1_{x_0,\GM}} \\
&\leq  \frac{1}{m} \left\| u_n(t_0) \right\|_{L^\infty(M)}^m+\frac{C}{t_0^{\frac{m}{m-1}}}  \, R^{\frac{2sm}{m-1}}  +  \dfrac{C}{t_0} \left\| u_{0,n} \right\|_{L^1_{x_0,\GM}}
.
\end{align*}
Taking the supremum over $x_0$ and raising to the power $ \frac 1 m $ yields
\[
\left\| u_n(t_0) \right\|_{L^\infty(M)} \le \frac{C}{t_0^{\frac{1}{m-1}}}  \, R^{\frac{2s}{m-1}}  +  \dfrac{C}{t_0^{\frac1m}} \left\| u_{0,n} \right\|_{L^1_{\GM}}^{\frac1m} \le \frac{C}{t_0^{\frac{1}{m-1}}}  \, R^{\frac{2s}{m-1}}  +  \dfrac{C}{t_0^{\frac1m}} \left\| u_{0} \right\|_{L^1_{\GM}}^{\frac1m} ,
\]
whence, letting $ n \to \infty $,
$$
\left\| u(t_0) \right\|_{L^\infty(M)} \le \liminf_{n \to \infty} \left\| u_n(t_0) \right\|_{L^\infty(M)} \le \frac{C}{t_0^{\frac{1}{m-1}}}  \, R^{\frac{2s}{m-1}}  +  \dfrac{C}{t_0^{\frac1m}} \left\| u_{0} \right\|_{L^1_{\GM}}^{\frac1m} ,
$$
still for all $ R \ge 1 $. The best choice in such bound is clearly $  R=1 $, which gives
\[
\begin{aligned}
\left\| u(t_0) \right\|_{L^\infty(M)} &\le C \left(\frac{1}{t_0^{\frac{1}{m-1}}}  +  \dfrac{\left\| u_0\right\|_{L^1_{\GM}}^{\frac1m}}{t_0^{\frac1m}} \right) \\
&\le \dfrac{C}{t_0^{\frac1m}} \left\| u_0\right\|_{L^1_{\GM}}^{\frac1m} \left[\frac1{\left(t_0^{\frac1{m-1}}\left\| u_0\right\|_{L^1_{\GM}}\right)^\frac1m}+1\right] \le \dfrac{C}{t_0^{\frac1m}} \left\| u_0\right\|_{L^1_{\GM}}^{\frac1m}
\end{aligned}
\]
provided $t_0\ge\|u_0\|_{{L^1_{\GM}}}^{-(m-1)}$, that is \eqref{thm.smoothing.HN-like.estimate4}. \qed

\section{Open problems}\label{open}

We conclude the paper by listing some open problems that might turn out to be interesting directions for future research.

\begin{itemize}
\item \textit{Solutions that may change sign: }Extend our results to signed solutions. Extension methods as in \cite{BGS} might be useful to this end.
\item \textit{Uniqueness: }Show that WDS are unique, not only the ones obtained by limits of monotone approximations, as done here. Such result is known from \cite{GMP1} in the Euclidean case for very weak solutions.
\item\textit{Mass conservation: }For positive, integrable solutions to \eqref{NFDE}, prove that $\|u(t)\|_1=\|u(0)\|_1$ for all such solutions and all $t>0$. Precise bounds for the fractional Laplacian of a test function should be proved, which is not elementary on general manifolds.
\item \textit{Large time behaviour: }Prove existence of fundamental solutions, namely positive solutions taking a Dirac delta as initial datum, and investigate their role in the asymptotic behaviour of general solutions as holds in the Euclidean case: see \cite{Vjems}, \cite{V} for the Euclidean non-fractional case, and as concerns existence a uniqueness only and in the non-fractional case again, and \cite{GMP2} in negatively curved manifolds.
\item \textit{Global Harnack Principle and convergence in relative error: }Prove (explicit) pointwise upper and lower bounds for solutions in the spirit  of the results in the euclidean setting: \cite{BV3} in fractional fast diffusive range $m<1$, \cite{BSV} for the fractional heat equation, $m=1$, and \cite{BS}, for the local fast diffusion equation, possibly with weights. Show that this implies convergence in relative error with respect to the fundamental solution with the same mass, as in \cite{BS}, and possibly characterize the class of data for which the GHP and the convergence in relative error holds.

\end{itemize}


\par\bigskip\noindent
\textbf{Acknowledgments.}
The first, third and fourth authors are members of the Gruppo Nazionale per l'Analisi Matematica, la Probabilit\`a e le loro Applicazioni (GNAMPA, Italy) of the Istituto Nazionale di Alta Matematica (INdAM, Italy) and are partially supported by the PRIN project 201758MTR2: ``Direct and Inverse Problems for Partial Differential Equations: Theoretical Aspects and Applications'' (Italy). The second author is partially supported by the Project PID2020-113596GB-I00 (Spain) and acknowledges financial support from the Spanish Ministry of Science
and Innovation, through the ``Severo Ochoa Programme for Centres of Excellence in R\&D'' (CEX2019-000904-S) and by the E.U. H2020 MSCA programme, grant agreement 777822.

\

\noindent\textbf{Conflict of interest statement. }On behalf of all authors, the corresponding author states that there is no conflict of interest.

\

\noindent\textbf{Data Availability Statements. }All data generated or analysed during this study are included in this published article

\



\begin{thebibliography}{99}

\bibitem{B} A.V. Balakrishnan, {\em An operational calculus for infinitesimal generators of semigroups}, Trans. Amer. Math. Soc. 91 (1959), 330--353.

\bibitem{BGFW} C. Bandle, M.d.M. Gonz\'alez, M.A. Fontelos, N. Wolanski, \it A nonlocal diffusion problem on manifolds\rm, Comm. Partial Differential Equations  43 \rm (2018), 652--676.

\bibitem{BGS} V. Banica, M.d.M. Gonz\'alez, M. S\'aez, \it Some constructions for the fractional Laplacian on noncompact manifolds\rm, Rev. Mat. Iberoam. 31 \rm (2015), 681--712.

\bibitem{BCr}{\rm  P. B\'enilan, M.G. Crandall}, \textit{Regularizing effects of homogeneous evolution equations}, Contributions to analysis and geometry (Baltimore, Md., 1980), pp. 23--39, Johns Hopkins Univ. Press, Baltimore, Md., 1981.


\bibitem{BBGG}  E. Berchio, M. Bonforte, D. Ganguly, G. Grillo, \it The fractional porous medium equation on the hyperbolic space\rm, Calc. Var. Partial Differential Equations  59 \rm (2020), Paper No.~169, 36 pp.

\bibitem{BFR}M. Bonforte, A. Figalli, X. Ros-Oton, \it Infinite speed of propagation and regularity of solutions to the fractional porous medium equation in general domains\rm, Comm. Pure Appl. Math. 70 \rm (2017), 1472--1508.

\bibitem{BFV} M. Bonforte, A. Figalli, J.L.~V\'azquez, \it Sharp global estimates for local and nonlocal porous medium-type equations in bounded domains\rm, Anal. PDE  11 \rm(2018), 945--982.


\bibitem{BGV} M.~Bonforte,   G.~Grillo, J.L.~V\'azquez, \it Fast diffusion flow on manifolds of nonpositive curvature\rm, J. Evol. Equ. 8 \rm (2008), 99--128.
\bibitem{BS}M. Bonforte, N. Simonov \emph{Fine properties of solutions to the Cauchy problem for a Fast Diffusion Equation with Caffarelli-Kohn-Nirenberg weights. }Preprint (2020).  \texttt{https://arxiv.org/abs/2002.09967}


\bibitem{BSV} M.~Bonforte, Y. Sire, J.L.~V\'azquez, \it Existence, uniqueness and asymptotic behaviour for fractional porous medium equations on bounded domains\rm, Discrete Contin. Dyn. Syst. 35 \rm (2015), 5725--5767.

\bibitem{BV3} M.~Bonforte,  J.L.~V\'azquez, \it Quantitative local and global a priori estimates for fractional nonlinear diffusion equations\rm, Adv. Math.  250 \rm (2014), 242--284.

\bibitem{BV2} M.~Bonforte,  J.L.~V\'azquez,  \emph{A priori estimates for fractional nonlinear degenerate diffusion equations on bounded domains}, Arch. Ration. Mech. Anal. 218 \rm (2015), 317--362.

\bibitem{BV1} M.~Bonforte,  J.L.~V\'azquez, \emph{Fractional nonlinear degenerate diffusion equations on bounded domains part I. Existence, uniqueness and upper bounds}, Nonlinear Anal. 131 \rm (2016), 363--398.



\bibitem{CS} L. Caffarelli, L. Silvestre, \it An extension problem related to the fractional Laplacian\rm, Comm. Partial Differential Equations 32 \rm (2007), 1245--1260.

\bibitem{C} G. Carron, \emph{In\'egalit\'es isop\'rim\'etriques de Faber-Krahn et cons\'equences} (French), Actes de la Table Ronde de G\'eom\'etrie Diff\'erentielle (Luminy, 1992), 205--232, S\'emin. Congr., 1, Soc. Math. France, Paris, 1996.



\bibitem{Caselli} M. Caselli, L. Gennaioli, \emph{Asymptotics as $s \rightarrow 0^+$ of the fractional perimeter on Riemannian manifolds}, preprint arXiv: \url{https://arxiv.org/abs/2306.11590}.

\bibitem{Caselli2} M. Caselli, E. Florit-Simon, J. Serra, \emph{Yau's conjecture for nonlocal minimal surfaces}, preprint arXiv \url{https://arxiv.org/abs/2306.07100}.


\bibitem{CM} T. Coulhon, S. Meda, \emph{Subexponential ultracontractivity and $L^p$-$L^q$ functional calculus}, Math. Z. 244 (2003), 291--308.




\bibitem{CR} M.G. Crandall, \emph{An introduction to evolution governed by accretive operators}, Dynamical systems (Proc. Internat. Sympos., Brown Univ., Providence, R.I., 1974), Vol. I, pp. 131--165, Academic Press, Inc. [Harcourt Brace Jovanovich, Publishers], New York-London, 1976.


\bibitem{CP} M.G.~Crandall, M.~Pierre, \emph{Regularizing effects for $u_t = A\varphi(u)$  in $L^1$},  J.~Funct. Anal. 45 (1982), \rm 194--212.

\bibitem{DA}  E.B.~Davies, ``Heat Kernels and Spectral Theory''. Cambridge Tracts in Mathematics, 92. Cambridge University Press, Cambridge, 1989.

\bibitem{dPQRV1} A. de Pablo, F. Quir\'os, A. Rodr\'iguez, J.L. V\'azquez, \it A fractional porous medium equation\rm, Adv. Math. 226 \rm (2011), 1378--1409.

\bibitem{dPQRV2} A. de Pablo, F. Quir\'os, A. Rodr\'iguez, J.L. V\'azquez, \it A general fractional porous medium equation\rm, Comm. Pure Appl. Math. 65 \rm (2012), 1242--1284.



    \bibitem{G} A. Grigor'yan, \emph{Heat kernels on weighted manifolds and applications. The ubiquitous heat kernel}, 93--191, Contemp. Math., 398, Amer. Math. Soc., Providence, RI, 2006.



\bibitem{GM1} G. Grillo, M. Muratori, \emph{Radial fast diffusion on the hyperbolic space}, Proc. London Math. Soc. 109 \rm (2014), 283--317.

\bibitem{GM} G. Grillo, M. Muratori, \emph{Smoothing effects  for the porous medium equation on Cartan-Hadamard manifolds},  Nonlinear Anal. 131 \rm (2016), 346--362.

\bibitem{GMP1} G. Grillo, M. Muratori, F. Punzo, \it Weighted fractional porous media equations: existence and uniqueness of weak solutions with measure data\rm, Calc. Var. Partial Differential Equations 54 \rm (2015), 3303--3335.

\bibitem{GMP2} G. Grillo, M. Muratori, F. Punzo, \emph{The porous medium equation with measure data on negatively curved Riemannian manifolds}, J. Eur. Math. Soc. (JEMS) 20 \rm (2018), 2769--2812.

\bibitem{GMPo} G. Grillo, M. Muratori, M.M. Porzio, \emph{Porous media equations with two weights: smoothing and decay properties of energy solutions via
 Poincar\'e inequalities}, Discr. Contin. Dynam. Systems  33 \rm (2013), 3599-3640.

\bibitem{GMP} G. Grillo, M. Muratori, F. Punzo, \emph{The porous medium equation with large initial data on negatively curved Riemannian manifolds}, J. Math. Pures Appl. 113 \rm (2018), 195--226.

\bibitem{GMV} G.~Grillo, M.~Muratori, J.L.~ V\'azquez,  \emph{The porous medium equation on Riemannian manifolds with negative
	curvature. The large-time behaviour},  Adv. Math. 314 \rm (2017), 328--377.

\bibitem{GMV-MA} G. Grillo, M. Muratori, J.L.~V\'azquez, \emph{The porous medium equation on Riemannian manifolds with negative curvature: the superquadratic case}, Math. Ann. 373 \rm (2019), 119--153.

\bibitem{H}  E. Hebey, ``Nonlinear Analysis on Manifolds: Sobolev Spaces and Inequalities''. Courant Lecture Notes in Mathematics, 5. New York University, Courant Institute of Mathematical Sciences, New York; American Mathematical Society, Providence, RI, 1999.

\bibitem{J}  N. Jacob, ``Pseudo Differential Operators and Markov Processes''. Vol. I. Fourier Analysis and Semigroups. Imperial College Press, London, 2001.


\bibitem{J-3}  N. Jacob, ``Pseudo Differential Operators and Markov Processes''. Vol. III. Markov Processes and Applications. Imperial College Press, London, 2005.


\bibitem{K} H. Komatsu, {\em Fractional powers of operators}, Pacific J. Math. 19 (1966), 285--346.

\bibitem{LY}  P. Li, S.-T. Yau, \it On the parabolic kernel of the Schr\"odinger operator\rm, Acta Math. 156 \rm (1986), 153--201.

 \bibitem{RS} N. Roidos, Y. Shao, \it The fractional porous medium equation on manifolds with conical singularities\rm, Evol. Equ. Control Theory, to appear, doi: 10.3934/eect.2021026.

 \bibitem{RS2} N. Roidos, E. Schrohe, {\em Existence and maximal $L^p$-regularity of solutions for the porous medium equation on manifolds with conical singularities}, Comm. Partial Differential Equations {41}  (2016), 1441--1471.

\bibitem{RS3} N. Roidos, E. Schrohe, {\em Smoothness and long time existence for solutions of the porous medium equation on manifolds with conical singularities}, Comm. Partial Differential Equations {43}  (2018), 1456--1484.

 \bibitem{Sti} P.R. Stinga, {\em User's guide to the fractional Laplacian and the method of semigroups}, in ``Handbook of fractional calculus with applications''. Vol. 2, De Gruyter, Berlin, 2019, pp. 235--265.

\bibitem{Stri} R.S. Strichartz, {\em Analysis of the Laplacian on the complete Riemannian manifold},
J. Funct. Anal. 52 (1983), 48--79.

\bibitem{Sturm} K.-T. Sturm, {\em Heat kernel bounds on manifolds}, Math. Ann. 292 \rm (1992), 149--162


\bibitem{Var} N.T. Varopoulos,  {\em Hardy–Littlewood theory for semigroups}, J. Funct. Anal. 63 \rm (1985), 240--260.


\bibitem{V} J.L. V{\'a}zquez, {``The Porous Medium Equation. Mathematical Theory''}. Oxford Mathematical Monographs. The Clarendon Press, Oxford University Press, Oxford, 2007.


\bibitem{V2} J.L. V{\'a}zquez, {``Smoothing and Decay Estimates for Nonlinear Diffusion Equations. Equations of Porous Medium Type''}. Oxford Lecture Series in Mathematics and its Applications, 33. Oxford University Press, Oxford, 2006.


\bibitem{Vjems} J.L. V{\'a}zquez, {\em Barenblatt solutions and asymptotic behaviour for a nonlinear fractional heat equation of porous medium type}, J. Eur. Math. Soc. 16 (2014),  769–803.

\bibitem{V4} J.L. V{\'a}zquez, {\em Fundamental solution and long time behaviour of the porous medium equation in hyperbolic space}, J. Math. Pures Appl. 104 \rm (2015), 454--484.


\end{thebibliography}
\end{document}